\documentclass[pdflatex,sn-mathphys-num]{sn-jnl}

\usepackage{graphicx}%
\usepackage{multirow}%
\usepackage{amsmath,amssymb,amsfonts}%
\usepackage{amsthm}%
\usepackage{mathrsfs}%
\usepackage[title]{appendix}%
\usepackage{xcolor}%
\usepackage{textcomp}%
\usepackage{manyfoot}%
\usepackage{booktabs}%
\usepackage{algorithm}%
\usepackage{algorithmicx}%
\usepackage{algpseudocode}%
\usepackage{listings}%
\usepackage{algpseudocode}

\usepackage{array}
\usepackage{makecell}
\usepackage{rotating}
\usepackage{multirow}
\usepackage{float}
\usepackage{enumerate}
\usepackage{subcaption}
\usepackage{mathptmx}
\usepackage{mathrsfs}

\usepackage[mathcal]{euscript}

\usepackage[labelfont=bf,   
            labelsep=space 
           ]{caption}

\newtheorem{thm}{Theorem}[section]

\newtheorem{lem}[thm]{Lemma}
\newtheorem{prop}[thm]{Proposition}
\newtheorem{ass}[thm]{Assumption}

\theoremstyle{definition}

\theoremstyle{remark}
\newtheorem{rem}[thm]{Remark}                                   
\numberwithin{equation}{section}

\begin{document}

\title[A linear convergence result for the Jacobi-Proximal Alternating Direction Method of Multipliers]{A linear convergence result for the Jacobi-Proximal Alternating Direction Method of Multipliers}

\author*{\fnm{Hyelin} \sur{Choi}}\email{hlchoimath@gmail.com}

\author{\fnm{Woocheol} \sur{Choi}}\email{choiwc@skku.edu}

\affil*{Department of Mathematics, Sungkyunkwan University, Suwon 16419, Republic of Korea}
\affil{Department of Mathematics, Sungkyunkwan University, Suwon 16419, Republic of Korea}

\abstract{This paper is concerned with the Jacobi-Proximal Alternating Direction Method of Multipliers (Jacobi-Proximal ADMM) proposed in the works He, Xu, and Yuan~(J. Sci. Comput., 2015)~\cite{On_the_proximal} and Deng et al.~(J. Sci. Comput., 2017)~\cite{Parallel_multi-block}, a parallel implementation of the ADMM for large-scale multi-block optimization problems  with a linear constraint.
While the $\mathit{o}(1/k)$ convergence of the Jacobi-Proximal ADMM for the case $N \geq 3$ has been well established in previous work, to the best of our knowledge, its linear convergence for $N \geq 3$ remains unproven. We establish the linear convergence of the algorithm when the cost functions are strongly convex and smooth. Numerical experiments are presented supporting the convergence result.}

\maketitle
\section{Introduction}\label{sec-1}
We consider the following convex optimization problem with $N$-block variables involving a linear constraint given as follows:
\begin{equation}\label{eq-1-19}
\begin{split}
\min_{x_1, \cdots, x_N} & \sum_{i=1}^N f_i(x_i)
\\
\textrm{subject to}~ &\sum_{i=1}^N A_ix_i = c,
\end{split}
\end{equation}
where $N \geq 2$, $x_i \in \mathbb{R}^{n_i},~ A_i \in \mathbb{R}^{m \times n_i},~c \in \mathbb{R}^m$ and $f_i :\mathbb{R}^{n_i} \to \left(-\infty, +\infty \right]~ (i=1,2,\cdots,N)$. 

The problem \eqref{eq-1-19} appears in various applications such as the image alignment problem in \cite{robust_alignment}, the robust principal component analysis model with noisy and incomplete data in \cite{recovering_lowrank}, the latent variable Gaussian graphical model selection in \cite{latent_variable,node-based_learning} and the quadratic discriminant analysis model in \cite{discriminant_analysis}.

To solve the problem \eqref{eq-1-19}, various algorithms have been proposed. 
A simple distributed algorithm for solving \eqref{eq-1-19} is dual decomposition \cite{Generalized_Lagrange}, which is essentially a dual ascent method or dual subgradient method \cite{Minimization_methods}. 
To describe it, we consider the Lagrangian for problem \eqref{eq-1-19}:
\begin{equation}\label{eq-1-16}
    \mathcal{L}(x_1, \cdots, x_N, \lambda) = \sum_{i=1}^N f_i(x_i) - \bigg\langle\lambda, \sum_{i=1}^N A_ix_i -c \bigg\rangle
\end{equation}
where $\lambda \in \mathbb{R}^m$ is the dual variable. Then, the dual decomposition updates as follows: for $k \geq 1$,
\begin{equation*}
\left\{
\begin{array}{ll}
     \left(x_1^{k+1}, x_2^{k+1},\cdots, x_N^{k+1} \right) = \arg \min_{\{x_i\}}\mathcal{L}(x_1, \cdots, x_N, \lambda^k),
     \\
     \lambda^{k+1} = \lambda^k - \alpha_k \left(\sum_{i=1}^N A_ix_i^{k+1}-c \right),
\end{array}
\right.
\end{equation*}
where $\alpha_k > 0$ is a step-size. 
Since all the $x_i$'s are separable in the Lagrangian function \eqref{eq-1-16}, the $x$-update step reduces to solving $N$ individual $x_i$- subproblems:
\begin{equation*}
    x_i^{k+1} = \arg\min_{x_i} f_i(x_i) - \big\langle \lambda^k, A_ix_i \big\rangle \quad \textrm{for} ~ i=1,2,\cdots, N,
\end{equation*}
and thus they can be carried out in parallel. 
With a suitable choice of $\alpha_k$ and certain assumptions, dual decomposition is guaranteed to converge to an optimal solution \cite{Minimization_methods}. 
However, dual decomposition exhibits slow convergence in practice, with a convergence rate of $\mathit{O}(1/\sqrt{k})$ for general convex problems.

An effective distributed algorithm for \eqref{eq-1-19} can be designed based on the alternating direction method of multipliers (ADMM) \cite{A_dual,Sur_l} using the following augmented Lagrangian:
\begin{equation*}
    \mathcal{L}_{\rho}(x_1, \cdots, x_N; \lambda) := \sum_{i=1}^N f_i(x_i) - \bigg \langle\lambda, \sum_{i=1}^N A_ix_i - c \bigg \rangle + \frac{\rho}{2}\bigg\|\sum_{i=1}^N A_ix_i -c\bigg\|^2,
\end{equation*}
where  $\rho > 0$ is a penalty parameter. 
In the context of ADMM using the augmented Lagrangian, the primal variable $x_i$ can be updated using different approaches. 
One common method is the Gauss-Seidel update, in which the primal variable $x_i$ is updated sequentially using the most recent values.
This scheme is commonly referred to as the Gauss-Seidel ADMM, and is given by:
\begin{equation*}
\begin{split}
x_i^{k+1} &= \arg\min_{x_i} \mathcal{L}_{\rho}(x_1^{k+1}, \cdots, x_{i-1}^{k+1}, x_i, x_{i+1}^k, \cdots, x_N^k, \lambda^k)
\\
& = \arg\min_{x_i}f_i(x_i) + \frac{\rho}{2} \bigg\|\sum_{j < i} A_jx_j^{k+1} + A_ix_i + \sum_{j > i} A_jx_j^k -c-\frac{\lambda^k}{\rho}\bigg\|_2^2.
\end{split}
\end{equation*}
The Gauss-Seidel ADMM was proposed by Wei and Ozdaglar \cite{Distributed_alternating}, where the authors established a convergence rate of $\mathit{O}(1/k)$ under the assumption that each objective function $f_i$ is convex.
Alternatively, the Jacobi-type update, commonly referred to as the Jacobi ADMM, provides a parallel update scheme for the primal variable $x_i$ as follows:
\begin{equation*}
\begin{split}
x_i^{k+1} &= \arg\min_{x_i} \mathcal{L}_{\rho}(x_1^{k}, \cdots, x_{i-1}^{k}, x_i, x_{i+1}^k, \cdots, x_N^k, \lambda^k)
\\
& = \arg\min_{x_i}f_i(x_i) + \frac{\rho}{2} \bigg\| A_ix_i + \sum_{j \neq i}A_jx_j^k -c-\frac{\lambda^k}{\rho}\bigg\|_2^2.
\end{split}
\end{equation*}
This scheme is suitable for parallel computation, and thus particularly effective for large-scale or big-data scenarios for the model \eqref{eq-1-19}.
Deng et al.~\cite{Parallel_multi-block} established an $\mathit{o}(1/k)$ convergence rate for the Jacobi ADMM, assuming that the objective functions $f_i$ are convex and that the matrices $A_i$ are full column-rank and mutually near-orthogonal.
However, with the same penalty parameter $\rho$, the Jacobi ADMM, which updates $x_i$ in parallel using previous-iterate information, typically exhibits inferior performance compared to the Gauss-Seidel ADMM, which updates blocks sequentially with the latest iterates.
In particular, the work \cite{On_full} found an example for which Jacobi ADMM diverges even when $N=2$.  
This issue was addressed by He, Xu, and Yuan~\cite{On_the_proximal}, who proposed the Jacobi-Proximal ADMM.
To ensure convergence, they introduced one key modification to the Jacobi ADMM:
\begin{itemize}
    \item \textbf{Proximal term $\frac{s\rho}{2} \|A_i(x_i - x_i^k)\|^2$:} 
    A proximal term $\frac{s\rho}{2} \|A_i(x_i - x_i^k)\|^2$ was added to each $x_i$-subproblem, where $\rho>0$ is the penalty parameter and $s>0$ is the proximal coefficient that enforces the new iterate to remain sufficiently close to the previous one.
\end{itemize}
Deng et al.~\cite{Parallel_multi-block} later studied the algorithm in a slightly different parameter setting, which is essentially the same as that of He, Xu, and Yuan~\cite{On_the_proximal}. 
This variant incorporates the following two modifications to the Jacobi ADMM:
\begin{itemize}
    \item \textbf{Proximal term $\mathbf{\frac{1}{2}\|x_i-x_i^k\|_{P_i}^2}$:} A proximal term $\frac{1}{2}\|x_i-x_i^k\|_{P_i}^2$ was added to each $x_i$-subproblem, where $P_i$ is a positive semi-definite matrix.
    This promotes strong convexity of the $x_i$-subproblems, thereby guaranteeing the uniqueness of their solutions even when the original objective functions are not strictly convex.
    \item \textbf{Damping parameter $\mathbf{\gamma >0}$:} The dual update
    \begin{equation}
    \lambda^{k+1} = \lambda^k - \rho \bigg(\sum_{i=1}^NA_ix_i^{k+1}-c\bigg)
    \end{equation} 
    was modified to
    \begin{equation}\label{eq-1-21}
    \lambda^{k+1} = \lambda^k - \gamma \rho\bigg(\sum_{i=1}^N A_ix_i^{k+1}-c \bigg),
    \end{equation}
    where $\gamma >0$ is a damping parameter used to control the step size and improve the algorithm's stability.
\end{itemize}

The Jacobi-Proximal ADMM is formulated as in Algorithm \ref{algo}.
\begin{algorithm}
\caption{Jacobi-Proximal ADMM}\label{algo}
\begin{algorithmic}[1]
\State \textbf{Initialize:} $x_i^0\ (i=1,2,\ldots,N)$ and $\lambda^0$;
\For{$k=0,1, \ldots$}
	\State Update $x_i$ for $i=1,\ldots,N$ in parallel by:
\begin{align}\label{eq-1-5}
&x_i^{k+1} =\arg\min_{x_i} \left\{f_i(x_i) + \frac{\rho}{2}\Bigg\|A_ix_i + \sum_{j \neq i}A_jx_j^k-c-\frac{\lambda^k}{\rho}\Bigg\|_2^2 + \frac{1}{2}\big\|x_i-x_i^k\big\|_{P_i}^2\right\};
\\\label{eq-1-6}
& \textrm{Update } \lambda^{k+1} = \lambda^k -\gamma \rho \bigg(\sum_{i=1}^N A_ix_i^{k+1} -c\bigg).
\end{align}
\EndFor
\end{algorithmic}
\end{algorithm}
Here we denote $\|x_i\|^2_{P_i}:= x_i^TP_ix_i$. 
The Jacobi-Proximal ADMM has the advantage of being highly parallelizable, as the updates for $x_i$ for each $i=1,2,\ldots, N$ can be computed simultaneously for each $i$, offering potential computational efficiency gains.
We note that the original $x_i$-subproblem \eqref{eq-1-5} contains a quadratic term of the form $\frac{\rho}{2} x_i^TA_i^TA_ix_i$.
If $A_i^TA_i$ is ill-conditioned or computationally expensive to invert, one approach is to modify $P_i$ as $P_i = D_i - \rho A_i^TA_i$, where $D_i$ is a well-conditioned and simple matrix(e.g., a diagonal matrix).
This modification eliminates the quadratic term $\frac{\rho}{2} x_i^TA_i^TA_ix_i$, making the $x_i$-subproblem easier to solve.
There are various choices for $P_i$ that can simplify the computation of the $x_i$-subproblem. 
Here are two widely used formulations for $P_i$ in \cite{On_the_global}, corresponding to the following methods:
\begin{enumerate}
    \item The \textit{standard proximal method}, where $P_i$ is taken as $P_i = \tau_i \mathbf{I}$, with $\tau_i > 0$.
    \item The \textit{prox-linear method}, where $P_i$ is taken as $P_i = \tau_i \mathbf{I} - \rho A_i^TA_i$, with $\tau_i > 0$.
\end{enumerate}
The work of Deng et al.\cite{Parallel_multi-block} established an $\mathit{o}(1/k)$ convergence result for the convergence of the Jacobi-Proximal ADMM when the proximal matrices $P_i$ satisfy $P_i \succ (\frac{N}{2-\gamma}-1)A_i^TA_i$ for $i=1,2, \ldots, N$ with $\gamma < 2$.
In particular, under specific choices of $P_i$, the convergence conditions reduce to $\tau_i > \rho (\frac{N}{2-\gamma} - 1) \|A_i\|^2$ for the \textit{standard proximal method} and $\tau_i > \frac{\rho N}{2-\gamma} \|A_i\|^2$ for the \textit{prox-linear method}.

In this work, we prove that the Jacobi-Proximal ADMM converges linearly under the assumption that each cost function $f_j$ is strongly convex. 
We derive a condition on the parameters $\rho > 0$ and $\gamma > 0$ for linear convergence, and we establish a linear convergence rate dependent on these parameters. We refer to Section \ref{sec-2} for the details and we provide a comparison of our convergence result with previous work on the problem \eqref{eq-1-19} in Table \ref{conv_result}.

\begin{sidewaystable}[htbp]
\centering
\begin{tabular}{@{} c c c c c c @{}}
\toprule
\makecell{} & 
\makecell{Ref.} & 
\makecell{Convexity} & 
\makecell{Regularity} & 
\makecell{Rank \\ Conditions} & 
\makecell{Rate} \\ 
\midrule
\makecell{Original ADMM \\ for $N=2$} &
\makecell{\cite{Local_linear}} &
\makecell{$f_1(x_1) + f_2(x_2$) \\ $=\frac{1}{2}x_1^TQx_1 + q^Tx_1+\frac{1}{2}x_2^TRx_2+r^Tx_2$} & 
\makecell{Local error bound} & 
\makecell{$B$: full column rank} &
\makecell{$\mathit{O}(\sigma^k)$(local)}  \\ 
\midrule
\makecell{Original ADMM \\ for $N=2$} &
\makecell{\cite{Linear_convergence_of_the}} &
\makecell{$f_1, f_2$: C} &
\makecell{Global error bound} &
\makecell{$A,B$: full column rank} &
\makecell{$\mathit{O}(\sigma^k)$} \\
\midrule
\makecell{Original ADMM \\ for $N=2$} & 
\makecell{\cite{Discerning_the}} & 
\makecell{Structured polyhedricity assumption \\ (see \cite[Assump.~1.3]{Discerning_the})} & 
\makecell{\textemdash} & 
\makecell{\textemdash} & 
\makecell{$\mathit{O}(\sigma^k)$(local)} \\ 
\midrule
\multirow{3}{*}{\makecell{Gauss-Seidel \\ ADMM \\ for $N \geq 3$}} & 
\multirow{3}{*}{\cite{On_the_global_linear}} & 
\makecell{$f_2,\ldots, f_N$: SC} & 
\makecell{$f_N$: $L$-smooth} & 
\makecell{$A_N$: full row rank} & 
\multirow{3}{*}{\makecell{$\mathit{O}(\sigma^k)$}} \\ 
\cmidrule{3-5}
 & 
 & 
\multicolumn{1}{c}{$f_1,\ldots, f_N$: SC} & 
\multicolumn{1}{c}{$f_1,\ldots, f_N$: $L$-smooth} & 
\multicolumn{1}{c}{\makecell{\textemdash}} & \\ 
\cmidrule{3-5}
 & 
 & 
\multicolumn{1}{c}{\makecell{$f_2,\ldots,f_N$: SC}} & 
\multicolumn{1}{c}{\makecell{$f_1,\ldots, f_N$: $L$-smooth}} & 
\multicolumn{1}{c}{\makecell{$A_1$: full column rank}} & \\ 
\midrule
\makecell{Gauss-Seidel \\ ADMM \\ for $N \geq 3$} & \makecell{\cite{On_the_linear}} & 
\makecell{$f_i(x_i) = g_i(E_ix_i) + h_i(x_i)$ \\ $g_i$: strictly C \\ $h_i$: C}& 
\makecell{
$g_i$: $L$-smooth \\
$\operatorname{epi}(h_i)$: polyhedral \\
$X_i$: compact polyhedral sets \\
Primal/Dual error bound
} & 
\makecell{$E_i$: full column rank} &
\makecell{$\mathit{O}(\sigma^k)$} \\ 
\midrule
\makecell{Jacobi \\ ADMM \\ for $N \geq 3$} &
\makecell{\cite{Parallel_multi-block}} &
\makecell{$f_1,\ldots,f_N$: C} &
\makecell{$A_i$: mutually near-orthogonal} &
\makecell{\textemdash} &
\makecell{$\mathit{o}(1/k)$} \\
\midrule
\makecell{Jacobi-Proximal\\ADMM\\for $N\geq 3$} & \makecell{\cite{Parallel_multi-block}} & 
\makecell{$f_1,\ldots,f_N$: C}& 
\makecell{\textemdash}  & 
\makecell{\textemdash} &
\makecell{$\mathit{o}(1/k)$} \\ 
\midrule
\makecell{Jacobi-Proximal\\ADMM\\for $N\geq 3$} & 
\makecell{This\\ work} & 
\makecell{$f_1,\ldots,f_N$: SC} & 
\makecell{$f_1,\ldots, f_N$: $L$-smooth} & 
\makecell{$\big[A_1^T;~ A_2^T;~ \cdots ;~A_N^T \big]$: full column rank} &
\makecell{$\mathit{O}(\sigma^k)$}  \\ 
\bottomrule
\end{tabular}
\vspace{0.1cm}
\caption[Summary of ADMM Linear Convergence Results]{ 
This table summarizes the convergence results for ADMM. 
In the `Convexity' column, C (resp., SC) means that the total cost function is assumed to be convex (resp., strongly convex). 
In the `Rate' column, $\mathit{O}(\sigma^k)$ denotes linear convergence with $0 < \sigma < 1$ ($\mathit{O}(\sigma^k)$(local) indicates local linear convergence, and $\mathit{o}(1/k)$ indicates convergence strictly faster than sublinear rate.
$N$ denotes the number of blocks of objective functions and $k$ denotes the number of iterations.
}\label{conv_result}
\end{sidewaystable}
\begin{sidewaystable}[htbp]
\centering
\begin{tabular}{@{} c c c c c c @{}}
\toprule
\makecell{} & 
\makecell{Ref.} & 
\makecell{Convexity} & 
\makecell{Regularity} & 
\makecell{Rank \\ Conditions} & 
\makecell{Rate} \\ 
\midrule
\makecell{Proximal \\ ADMM} & 
\makecell{\cite{Partial_error}} & 
\makecell{$f_1,f_2$:C} & 
\makecell{Strong Conical Hull \\Intersection Property} & 
\makecell{If $P_1 \succeq 0$: $A_1,A_2$: full column rank \\
If $P_1\succ 0$: $A_2$: full column rank} &
\makecell{$\mathit{O}(\sigma^k)$}  \\ 
\midrule
\multirow{5}{*}{\makecell{PADMM-FG}} & 
\multirow{5}{*}{\makecell{\cite{On_the_global}}} & 
\makecell{$f_1$: SC} & 
\makecell{$f_1$: $L$-smooth} & 
\makecell{$A_1$: full column rank \\ $A_2$: full row rank} & 
\multirow{5}{*}{\makecell{$\mathit{O}(\sigma^k)$}} \\ 
\cmidrule{3-5}
 & 
 & 
\multicolumn{1}{c}{\makecell{$f_1, f_2$: SC}} & 
\multicolumn{1}{c}{$f_1$: $L$-smooth} & 
\multicolumn{1}{c}{\makecell{$A_2$: full row rank}} & \\ 
\cmidrule{3-5}
 & 
 & 
\multicolumn{1}{c}{$f_1$: SC} & 
\multicolumn{1}{c}{$f_1,f_2$: $L$-smooth} & 
\multicolumn{1}{c}{\makecell{$A_1$: full column rank}} &  \\ 
\cmidrule{3-5}
 & 
 & 
\multicolumn{1}{c}{$f_1, f_2$: SC} & 
\multicolumn{1}{c}{$f_1, f_2$: $L$-smooth} & 
\multicolumn{1}{c}{\makecell{\textemdash}} & \\ 
\midrule
\multirow{3}{*}{\makecell{PADMM-FG}} & 
\multirow{3}{*}{\makecell{\cite{Discerning_the}}} & 
\makecell{Structured polyhedricity assumption \\ (see \cite[Assump.~1.3]{Discerning_the})} & 
\makecell{\textemdash} & 
\makecell{\textemdash} & 
\multirow{3}{*}{\makecell{$\mathit{O}(\sigma^k)$(local)}} \\ 
\cmidrule{3-5}
 & 
 & 
\multicolumn{1}{c}{\textemdash} & 
\multicolumn{1}{c}{\shortstack{Structured subregularity assumption \\ (see \cite[Assump.~1.4]{Discerning_the})}} & 
\multicolumn{1}{c}{\makecell{$A_1$: full column rank \\ $A_2$: full row rank}} & \\ 
\midrule
\makecell{Relaxed \\ ADMM} & 
\makecell{\cite{metric_selection}} & 
\makecell{$f_1$:SC} &
\makecell{$f_1$:$L$-smooth} &
\makecell{$A$: full row rank} &
\makecell{$\mathit{O}(\sigma^k)$} \\
\midrule
\multirow{5}{*}{\makecell{Relaxed \\ ADMM}} & 
\multirow{5}{*}{\makecell{\cite{Faster_convergence}}} & 
\makecell{$f_1$: SC} & 
\makecell{$f_1$: $L$-smooth} & 
\makecell{$A_1$: full row rank} & 
\multirow{5}{*}{\makecell{$\mathit{O}(e^{-\alpha k})$}} \\ 
\cmidrule{3-5}
 & 
 & 
\multicolumn{1}{c}{\makecell{$f_2$: SC}} & 
\multicolumn{1}{c}{$f_2$: $L$-smooth} & 
\multicolumn{1}{c}{\makecell{$A_2$: full row rank}} & \\ 
\cmidrule{3-5}
 & 
 & 
\multicolumn{1}{c}{$f_2$: SC} & 
\multicolumn{1}{c}{$f_1$: $L$-smooth} & 
\multicolumn{1}{c}{\makecell{$A_1$: full row rank}} &  \\ 
\cmidrule{3-5}
 & 
 & 
\multicolumn{1}{c}{$f_1$: SC} & 
\multicolumn{1}{c}{$f_2$: $L$-smooth} & 
\multicolumn{1}{c}{\makecell{$A_2$: full row rank}} & \\ 
\midrule
\makecell{Over-Relaxed \\ ADMM} & 
\makecell{\cite{A_general}} & 
\makecell{$f_1$:SC, $f_2$:C} & 
\makecell{$f_1$: $L$-smooth} & 
\makecell{$A_2$: full column rank} &
\makecell{$\mathit{O}(\sigma^k)$}  \\ 
\midrule
\multirow{2}{*}{\makecell{Linearized \\ADMM}} & 
\multirow{2}{*}{\makecell{\cite{Discerning_the}}} & 
\makecell{Structured polyhedricity assumption \\ (see \cite[Assump.~1.3]{Discerning_the})} & 
\makecell{\textemdash} & 
\makecell{\textemdash} & 
\multirow{3}{*}{\makecell{$\mathit{O}(\sigma^k)$(local)}} \\ 
\cmidrule{3-5}
 & 
 & 
\multicolumn{1}{c}{\textemdash} & 
\multicolumn{1}{c}{\shortstack{Structured subregularity assumption \\ (see \cite[Assump.~1.4]{Discerning_the})}} & 
\multicolumn{1}{c}{\makecell{$A_2$: full row rank}} & \\  
\bottomrule
\end{tabular}
\vspace{0.1cm}
\caption[Summary of ADMM Linear Convergence Results]{ 
This table summarizes the convergence results for several variants of ADMM in the case $N=2$. Here, \mbox{PADMM-FG} refers to the proximal ADMM with the Fortin–Glowinski relaxation step.
In the `Convexity' column, C (resp., SC) means that the total cost function is assumed to be convex (resp., strongly convex). 
In the `Rate' column, $\mathit{O}(\sigma^k)$ with $0 < \sigma < 1$ and $\mathit{O}(e^{-\alpha k})$ with $\alpha > 0$ both denote linear convergence, while $\mathit{O}(\sigma^k)$(local) indicates local linear convergence.
$N$ denotes the number of blocks of objective functions and $k$ denotes the number of iterations.
}\label{conv_result_variant}
\end{sidewaystable}

\subsection{Related works}\label{subsec-1}
The Alternating Direction Method of Multipliers (ADMM) \cite{A_dual,Sur_l} is an efficient approach that facilitates the solution of structured optimization problems by decomposing them into a sequence of simpler subproblems.
Building on the classical Douglas–Rachford operator splitting method \cite{On_the}, Glowinski and Marrocco \cite{Sur_l}, along with Gabay and Mercier \cite{A_dual}, introduced the well-known Alternating Direction Method of Multipliers (ADMM) with a Gauss–Seidel-type decomposition for solving \eqref{eq-1-19} in the case of $N=2$.
It employs the augmented Lagrangian for \eqref{eq-1-19}:
\begin{equation*}
    \mathcal{L}_{\rho}(x_1, \cdots, x_N; \lambda) := \sum_{i=1}^N f_i(x_i) - \left<\lambda, \sum_{i=1}^N A_ix_i - c \right> + \frac{\rho}{2}\left\|\sum_{i=1}^N A_ix_i -c\right\|^2
\end{equation*}
which incorporates a quadratic penalty term of the constraints (with a penalty parameter $\rho > 0$) into the Lagrangian. 
At each iteration, the augmented Lagrangian is sequentially minimized with respect to the primal variables $x_1$ and $x_2$, followed by an update of the dual variable $\lambda$.
The ADMM has been extensively investigated for solving two-block convex minimization problems (i.e., $N=2$). 
Its global convergence in the case of $N=2$ has been established in \cite{On_the_Douglas,Chapter_ix}. 
Han and Yuan \cite{Local_linear} showed that original ADMM achieves a local linear convergence rate for quadratic programs by using an error bound established in \cite{Error_bound}.
Based on this error bound framework, Yang and Han \cite{Linear_convergence_of_the} proved that ADMM attains global linear convergence for convex problems with polyhedral constraints by leveraging a global error bound.
Yuan, Zeng, and Zhang \cite{Discerning_the} showed that original ADMM also enjoys global linear convergence under the structured polyhedrality assumption by employing variational analysis techniques.
These scenarios are summarized in Table~\ref{conv_result}.
Additionally, various studies have explored the convergence rate properties of original ADMM when $N=2$ (see, e.g., \cite{Splitting_algorithms,Local_linear_linear,On_the_O(1/n),Iteration-complexity_of,On_non-ergodic}).
In addition to these results on the original ADMM, convergence of several variants for the case of $N=2$ has also been studied, such as the Proximal ADMM \cite{Partial_error}, Proximal ADMM with the Fortin–Glowinski relaxation step (PADMM-FG) \cite{On_the_global, Discerning_the}, Relaxed ADMM \cite{metric_selection, Faster_convergence}, Over-Relaxed ADMM \cite{A_general}, and Linearized ADMM \cite{Discerning_the}.
The convergence results of these variants are summarized in Table 2.

While the convergence results of ADMM is well established for the case $N=2$, its behavior for $N\geq 3$ has long remained unresolved; particularly, Chen et al. \cite{The_direct} provided a concrete example in which Gauss-Seidel ADMM diverges for $N=3$ for any penalty parameter $\rho$.
Building on this, we first discuss studies that establishes the convergence conditions for ADMM using the Gauss-Seidel scheme when $N \geq 3$, where the Gauss-Seidel method updates variables sequentially.
We refer to the works \cite{A_convergent}, \cite{convergence_analysis}, and \cite{Global_convergence} that achieved global and sublinear convergence results of the 3-block ADMM.
Han and Yuan \cite{A_note} established the global convergence of ADMM under the assumption that $f_1, \cdots, f_N$ are all strongly convex.
Building on this, both Chen, Shen, and You \cite{convergence_analysis} and Lin, Ma, and Zhang \cite{On_the_sublinear} relax the assumptions in \cite{A_note}.
In particular, \cite{convergence_analysis}, for $N=3$, allows one objective function to be not necessarily strongly convex and also relax the range of the penalty parameter, and propose a relaxed ADMM that incorporates an additional computation of optimal step size and prove its global convergence.
Lin, Ma, and Zhang \cite{On_the_sublinear} demonstrate sublinear convergence with rate $\textit{O}(1/k)$ in the ergodic sense and $\textit{o}(1/k)$ in the non-ergodic sense when one objective function is convex and the remaining $N-1$ are strongly convex, provided that the penalty parameter lies within a certain range.
Lin, Ma, and Zhang \cite{On_the_global_linear} proved the global linear convergence of Gauss-Seidel ADMM under the three scenarios, which are detailed in Table~\ref{conv_result}.
The global R-linear convergence of Gauss-Seidel ADMM has been established by Hong and Luo \cite{On_the_linear}; the corresponding condition is summarized in Table~\ref{conv_result}.

Alternatively, Jacobi updates in ADMM have been extensively studied, offering parallelization advantages for multi-block optimization problems.
The linear convergence of Proximal ADMM and modified Jacobi ADMM, as well as that of Gauss-Seidel ADMM, which was mentioned earlier, is established under several assumptions in~\cite{On_the_linear}.
Deng et al. \cite{Parallel_multi-block} showed that the Jacobi ADMM and Jacobi-Proximal ADMM achieve $\mathit{o}(1/k)$ convergence rate under the conditions summarized in Table~\ref{conv_result}.
He, Hou, and Yuan~\cite{On_full} proposed a variant of multi-block Jacobi ADMM by incorporating an under-relaxation step into the full Jacobian decomposition framework, achieving a worst-case convergence rate of $O(1/k)$ in both ergodic and non-ergodic senses.

\subsection{Contribution}\label{subsec-2}
The main contribution in this paper are as follows. We establish a global linear convergence of Jacobi-Proximal ADMM(Algorithm~\ref{algo}) under the assumption that the objective functions $f_i$ are strictly convex and $L_i$-smooth.
Our numerical results on the LCQP model and the optimal resource allocation problem support the linear convergence result of this work.

\subsection{Organization}\label{subsec-3}
We organize the rest of this paper as follows. 
In Section~\ref{sec-2}, we state the assumptions and the main result of this paper. 
Section~\ref{sec-3} is devoted to prove the main result.
In Section~\ref{sec-4}, we present numerical experiments supporting the linear convergence result.

\section{Statement of the main result}\label{sec-2}
In this section, we state the main result of this paper concerning the linear convergence of the Jacobi-Proximal ADMM described in Algorithm \ref{algo}. 
Since both of the objective function and constraints of \eqref{eq-1-19} are sums of terms on individual $x_i$'s, the problem can be decomposed into $N$ smaller subproblems, which is solved in a parallel and distributed manner.

Related to the problem \eqref{eq-1-19}, we first define the following augmented Lagrangian
\begin{equation*}
\mathcal{L}_{\rho}(x_1, \cdots, x_N; \lambda) := \sum_{i=1}^N f_i(x_i) - \left<\lambda, \sum_{i=1}^N A_ix_i - b \right> + \frac{\rho}{2}\left\|\sum_{i=1}^N A_ix_i -b\right\|^2
\end{equation*}
where $\lambda \in \mathbb{R}^m$ is the dual variable and parameter $\rho > 0$. 

Here we let $n = \sum_{i=1}^N n_i$ and set $A \in \mathbb{R}^{m \times n}$ by $A = \big[A_1 \cdots A_N\big]$. 
For $k \geq 0$, we denote by $\left( x_1^k, \cdots, x_N^k; \lambda^k \right)$ the variable generated by Algorithm \ref{algo}.
Also, we set $x^k \in \mathbb{R}^n$ for $k \geq 0$ and $x^* \in \mathbb{R}^n$ as $x^k = \big[x_1^k;\cdots ;x_N^k\big]$ and $x^* = \big[x_1^*; \cdots;x_N^*\big]$. 
With this notation, we have 
\begin{equation*}
Ax^k = \sum_{i=1}^N A_i x_i^k.
\end{equation*}
We use $\|z\|$ to denote the standard Euclidean norm of $z$, and $\langle \cdot, \cdot\rangle$ to denote the standard inner product. 
For a positive definite matrix $M \in \mathbb{R}^{l \times l}$, we define the $M$-norm as $\|z\|_M := \sqrt{z^TMz},\quad \forall z \in \mathbb{R}^l$.
If the matrix $M$ is positive semi-definite, then $\|\cdot\|_M$ is a semi-norm.
We state assumptions on the cost functions which will be used throughout the paper.
\begin{ass}\label{ass-1}
The functions $f_i: \mathbb{R}^{n_i} \rightarrow \left(-\infty,+\infty \right]$ are closed proper convex for $1\leq i\leq N$.
\end{ass}
We assume that there exists a minimizer of problem \eqref{eq-1-19} satisfying the KKT condition as follows.
\begin{ass}\label{ass-2}
We assume that $x^*= (x_1^*, \cdots, x_N^*) \in \mathbb{R}^{n}$ is a minimizer of problem \eqref{eq-1-19} companied with a multiplier $\lambda^* \in \mathbb{R}^{m}$ satisfying the following KKT condition 
\begin{align}\label{eq-1-10}
\nabla f_i\left(x_i^*\right) = A_i^T \lambda^* \quad \textrm{for}~ i=1,\cdots, N,
\\\label{eq-1-11}
Ax^* = \sum_{i=1}^N A_i x_i^* = c.
\end{align}
\end{ass}

\begin{ass}\label{ass-3}
There are positive constants $L_{i}$ such that for all $x,y \in \mathbf{R}^{n_i}$,
\begin{equation*}
\left\|\nabla f_i(x) - \nabla f_i(y) \right\| \leq L_i \left\| x-y\right\| \textrm{ for } 1\leq i \leq N.
\end{equation*}
\end{ass}
    
\begin{ass}\label{ass-4}
The function $f_i: \mathbb{R}^{n_i} \to (-\infty, +\infty ]$ is $\alpha$-strongly convex for some $\alpha >0$, $i.e.$,
\begin{equation*}
f_i(y) \geq f_i(x) + \left\langle y-x , \nabla f_i(x)\right\rangle + \alpha \left\| x-y\right\|^2
\end{equation*}
for all $x, y \in \mathbb{R}^{n_i}$.
\end{ass}

\begin{ass}\label{ass-5}We assume that the matrix $\big[A_1^T;~ A_2^T;~ \cdots ;~A_N^T \big] \in \mathbb{R}^{n \times m}$ with $n = \sum_{i=1}^N n_i$ has full column rank.
\end{ass}

We state the main result of this paper.
\begin{thm}\label{thm-main} Suppose Assumptions \ref{ass-1}-\ref{ass-5} hold and assume that $P_i$ is a positive semi-definite matrix for each $1\leq i \leq N$. 
Choose any $s >0$ such that
\begin{equation*}
s \Big[ \rho^2 D \|A_i \| + \frac{L}{N}\Big] < \frac{\alpha}{2N}, \qquad \forall 1 \leq i \leq N,
\end{equation*}  
and suppose that  $0<\gamma<2$, $\rho >0$  are chosen so that 
there exist values $\xi_i > 0$ such that
\begin{equation}\label{eq-2-2}
\left\{ \begin{array}{ll}
\rho A_i^TA_i + P_i -8s (\rho A_i^TA_i + P_i)^T(\rho A_i^TA_i + P_i) - \frac{\rho}{\xi_i}A_i^TA_i \succ 0,
 \\
\sum_{i=1}^N \xi_i < 2- \gamma. \end{array}\right.
\end{equation}
Define $\sigma = \max \left\{ {1-2\gamma\rho sc_A^2, \mu_s}\right\} \in (0,1)$ for any $c_A \in \Big(0, \frac{1}{\sqrt{2\gamma \rho s}}\Big)$, and $\mu_s \in (0,1)$ satisfying 
\begin{equation*}
\left\{  
\begin{aligned}
        &c_A^2 I_m \preceq \sum_{i=1}^N A_i A_i^T, 
        \\
        &(\rho + 4N s \rho^2D)A_i^TA_i + P_i \preceq \mu_s (\rho A_i^T A_i + P_i + 2(\alpha - Ls)I_{n_i}.
\end{aligned}
\right.
\end{equation*}
Then we have
\begin{equation}\label{eq-2-1}
\phi\big(x^{k+1}, \lambda^{k+1}\big) \leq \sigma \phi\big(x^k, \lambda^{k}\big),
\end{equation}
where $\phi: \mathbb{R}^n \times \mathbb{R}^m \to \mathbb{R}$ is defined as 
\begin{equation*}
    \phi \big(x^k, \lambda^k\big) = \frac{1}{2\gamma \rho} \big\|\lambda^k - \lambda^* \big\|^2 +  \frac{1}{2}   \sum_{i=1}^N \big\|x_i^{k} - x_i^{*} \big\|_{\rho A_i^TA_i + P_i +2 (\alpha -2Ls) I_{n_i}}^2.
\end{equation*}
As a result of \eqref{eq-2-1}, we have the following linear convergence result
\begin{equation*}
\phi \big(x^k, \lambda^k\big) = O(\sigma^k).
\end{equation*}
\end{thm}

\begin{rem}
Choosing $\xi_i < \frac{2-\gamma}{N}$ leads to a simplification of the convergence condition \eqref{eq-2-2}, resulting in the following form: 
\begin{equation}\label{eq-2-3}
\rho A_i^TA_i + P_i -8s (\rho A_i^TA_i + P_i)^T(\rho A_i^TA_i + P_i) - \frac{\rho N}{2-\gamma}A_i^TA_i \succ 0.
\end{equation}
In the case of $P_i = \tau_i \mathbf{I}$ (\textit{standard proximal method}), the condition \eqref{eq-2-3} further reduces to
\begin{equation*}
\rho\|A_i\|^2 + \tau_i -8s (\rho \|A_i\|^2 + \tau_i)^2 - \frac{\rho N}{2-\gamma}\|A_i\|^2 > 0.
\end{equation*}
Similarly, for $P_i = \tau_i \mathbf{I} - \rho A_i^TA_i$ (\textit{prox-linear method}), the condition \eqref{eq-2-3} simplifies to
\begin{equation*}
\tau_i - 8s\tau_i^2 - \frac{\rho N}{2-\gamma}\|A_i\|^2 > 0.
\end{equation*}
\end{rem}

\section{Proof of the main result}\label{sec-3}
In this section, we give the proof of Theorem \ref{thm-main}. For this, we prepare several lemmas. 

\begin{lem}\label{lem-2}
Suppose Assumption \ref{ass-5} holds.
Consider matrices $A_i \in \mathbb{R}^{m \times n_i}$ for $1 \leq i \leq N$.
Then there exists $c_A>0$ such that 
\begin{equation}\label{eq-1-14}
\bigg(\sum_{i=1}^N\|A_i^T \lambda\|^2\bigg)^{\frac{1}{2}} \geq c_A\|\lambda \|
\end{equation}
for all $\lambda \in \mathbb{R}^{m}$.
\end{lem}
\begin{proof}
We define $c_A \geq 0$ by
\begin{equation*}
c_A = \min_{\lambda \in \mathbb{R}^m \backslash \{0\}} \frac{\Big(\sum_{i=1}^N\|A_i^T \lambda\|^2\Big)^{\frac{1}{2}}}{\|\lambda\|} = \min_{\lambda \in S^{m-1}} \frac{\Big(\sum_{i=1}^N\|A_i^T \lambda\|^2\Big)^{\frac{1}{2}}}{\|\lambda\|}.
\end{equation*}
In order to prove that $c_A > 0$, we assume $c_A = 0$ and proceed to derive a contradiction.
Since $S^{m-1}$ is compact, there exists a minimizer $\lambda_0 \in S^{m-1}$ of \eqref{eq-1-14} such that
\begin{equation*}
\frac{\Big(\sum_{i=1}^N \big\|A_i^T\lambda_0\big\|^2\Big)^{\frac{1}{2}}}{\|\lambda_0\|}=0.
\end{equation*}
Thus we have $A_i^T \lambda_0 = 0$ for all $1 \leq i \leq N$, and this leads to $\lambda_0 = 0$ since $\big[A_1^T; ~\cdots; ~ A_N^T \big]$ has full column rank. It contradicts to the fact that $\lambda_0 \in S^{m-1}$. Therefore, we have $c_A > 0$, and the lemma is proved.
\end{proof}

\begin{rem}
The assumption of Lemma \ref{lem-2} means that the matrix $\big[A_1^T;~ A_2^T;~ \cdots ;~A_N^T \big] \in \mathbb{R}^{n \times m}$ with $n = n_1+ \cdots + n_N$ has only a trivial kernel, i.e., $\ker\left( \big[A_1^T;~ A_2^T;~ \cdots ;~A_N^T \big] \right) = \{\mathbf{0}\}$.
\end{rem}

By the strongly convexity of the objective function, we have the following lemma.
\begin{lem}\label{lem-3-2}
Suppose Assumptions \ref{ass-1}, \ref{ass-2}, and \ref{ass-4} hold. Then it holds that 
\begin{equation}\label{eq-1-20}
\begin{split}
&\frac{1}{2\gamma \rho}\big\|\lambda^k - \lambda^*\big\|^2 + \frac{1}{2}\sum_{i=1}^N \big\|x_i^{k}-x_i^*\big\|_{\rho A_i^TA_i+P_i}^2
\\
&\geq \alpha \sum_{i=1}^N \big\|x_i^{k+1}-x_i^*\big\|^2 + \frac{1}{2\gamma \rho}\big\|\lambda^{k+1} - \lambda^*\big\|^2+\frac{1}{2}\sum_{i=1}^N \big\|x_i^{k+1}-x_i^*\big\|_{\rho A_i^TA_i+P_i}^2 + H_k
\end{split}
\end{equation} 
where $H_k := \frac{1}{2}\sum_{i=1}^N \big\|x_i^{k+1}-x_i^k\big\|_{\rho A_i^TA_i+P_i}^2 + \frac{2-\gamma}{2\gamma^2\rho}\big\|\lambda^{k+1} - \lambda^{k}\big\|^2 + \frac{1}{\gamma} \Big\langle \lambda^k - \lambda^{k+1}, A\big(x^k - x^{k+1}\big)\Big\rangle$.
\end{lem}
\begin{proof}
By the optimality of \eqref{eq-1-5}, we have
\begin{equation}\label{eq-1-7}
\nabla f_i(x_i^{k+1}) = A_i^T \lambda^k - \rho A_i^T \bigg(A_ix_i^{k+1} + \sum_{j \neq i}A_jx_j^k-c\bigg) + P_i\big(x_i^k - x_i^{k+1}\big).
\end{equation}
Combining \eqref{eq-1-7} and \eqref{eq-1-10}, we get
\begin{equation}\label{eq-1-12}
\begin{split}
&A_i^T \big(\lambda^k - \lambda^* \big) 
\\
&= \nabla f_i\big(x_i^{k+1}\big) - \nabla f_i\left(x_i^*\right) + \rho A_i^T \bigg(A_ix_i^{k+1} + \sum_{j \neq i}A_jx_j^k-c \bigg) + P_i\big(x_i^{k+1} - x_i^k \big).
\end{split}
\end{equation}
Since each $f_i$ is $\alpha$-strongly convex, we have
\begin{equation}\label{eq-1-1}
\left\langle x_i^{k+1} - x_i^*, \nabla f_i\big(x_i^{k+1}\big) - \nabla f_i\big(x_i^*\big) \right\rangle \geq \alpha \big\|x_i^{k+1} - x_i^*\big\|^2.
\end{equation}
Inserting \eqref{eq-1-12} into \eqref{eq-1-1}, we have
\begin{equation*}
\begin{split}
\bigg\langle A_i\big(x_i^{k+1} - x_i^*\big), \lambda^k - \lambda^*- \rho \bigg(A_ix_i^{k+1} + \sum_{j \neq i}A_jx_j^k-c \bigg)\bigg\rangle 
\\
+ \left\langle x_i^{k+1} - x_i^*, P_i\big(x_i^k - x_i^{k+1}\big)\right\rangle \geq \alpha \big\|x_i^{k+1} - x_i^*\big\|^2.
\end{split}
\end{equation*}
We rewrite above inequality as follows
\begin{equation*}
\begin{split}
\bigg\langle A_i\big(x_i^{k+1} - x_i^*\big), \lambda^k - \lambda^*- \rho \bigg(\sum_{i=1}^N A_ix_i^{k+1} + \sum_{j \neq i}A_j\big(x_j^k-x_j^{k+1}\big)-c\bigg)\bigg\rangle 
\\
+ \left\langle x_i^{k+1} - x_i^*, P_i\big(x_i^k - x_i^{k+1}\big)\right\rangle \geq \alpha \big\|x_i^{k+1} - x_i^*\big\|^2.
\end{split}
\end{equation*}
For simplicity, we let $\hat{\lambda} = \lambda^k - \rho\left(\sum_{i=1}^N A_ix_i^{k+1}-c\right) = \lambda^k - \frac{1}{\gamma}\left(\lambda^k - \lambda^{k+1}\right)$. Then the above inequality reads as follows
\begin{equation*}
\begin{split}
\bigg\langle A_i\big(x_i^{k+1} - x_i^*\big), \hat{\lambda} - \lambda^*- \rho \Big(\sum_{j \neq i}A_j(x_j^k-x_j^{k+1})\Big)\bigg\rangle 
\\
+ \left\langle x_i^{k+1} - x_i^*, P_i\big(x_i^k - x_i^{k+1}\big)\right\rangle \geq \alpha \big\|x_i^{k+1} - x_i^*\big\|^2.
\end{split}
\end{equation*}
Using that $\sum_{j \neq i}A_j\big(x_j^k - x_j^{k+1}\big) = A\big(x^k - x^{k+1}\big) - A_i\big(x_i^k - x_i^{k+1}\big)$ and summing the above estimate for $1\leq i \leq N$ yields
\begin{equation}\label{eq-1-2}
\begin{split}
\left\langle A\big(x^{k+1} - x^*\big), \hat{\lambda}-\lambda^*\right\rangle+\sum_{i=1}^N\big(x_i^{k+1}-x_i^*\big)^T \left(\rho A_i^T A_i + P_i\right)\big(x_i^k - x_i^{k+1}\big) 
\\
\geq \alpha \sum_{i=1}^N \big \|x_i^{k+1}-x_i^*\big\|^2 + \rho \Big\langle A\big(x^{k+1} - x^*\big), A\big(x^k - x^{k+1}\big)\Big\rangle.
\end{split}
\end{equation}
Combining the update rule \eqref{eq-1-6} and \eqref{eq-1-11} and using that $\hat{\lambda} - \lambda^* = \left(\lambda^{k+1} - \lambda^*\right) + \frac{\gamma-1}{\gamma}\left(\lambda^k - \lambda^{k+1}\right)$ in the above inequality, we get 
\begin{equation}\label{eq-1-3}
\begin{split}
&\left\langle \frac{1}{\gamma \rho} \big(\lambda^k - \lambda^{k+1}\big), \lambda^{k+1} - \lambda^* +\frac{\gamma-1}{\gamma}\big(\lambda^k - \lambda^{k+1}\big)\right\rangle 
\\
& +\sum_{i=1}^N\big(x_i^{k+1}-x_i^*\big)^T \left(\rho A_i^T A_i + P_i\right)\big(x_i^k - x_i^{k+1}\big) 
\\
&\geq \alpha \sum_{i=1}^N \big\|x_i^{k+1}-x_i^*\big\|^2 + \frac{1}{\gamma} \left\langle \lambda^k - \lambda^{k+1}, A\big(x^k - x^{k+1}\big)\right\rangle.
\end{split}
\end{equation}
For the two terms in the left hand side, we have the following identities
\begin{equation*}
\begin{split}
&\left\langle \frac{1}{\gamma \rho} \big(\lambda^k - \lambda^{k+1}\big), \lambda^{k+1} - \lambda^* +\frac{\gamma-1}{\gamma}\big(\lambda^k - \lambda^{k+1}\big)\right\rangle 
\\
&= \frac{\gamma-2}{2\gamma^2\rho}\left\|\lambda^{k+1} - \lambda^{k}\right\|^2 - \frac{1}{2\gamma \rho}\left\|\lambda^{k+1}-\lambda^*\right\|^2 + \frac{1}{2\gamma \rho}\left\|\lambda^k - \lambda^*\right\|^2,
\end{split}
\end{equation*}
and
\begin{equation*}
\begin{split}
&\sum_{i=1}^N\big(x_i^{k+1}-x_i^*\big)^T \left(\rho A_i^T A_i + P_i\right)\big(x_i^k - x_i^{k+1}\big)
\\
&= -\frac{1}{2} \sum_{i=1}^N \Big( \big\|x_i^{k+1}-x_i^*\big\|_{\rho A_i^TA_i+P_i}^2+\big\|x_i^{k+1}-x_i^k\big\|_{\rho A_i^TA_i+P_i}^2 - \big\|x_i^{k}-x_i^*\big\|_{\rho A_i^TA_i+P_i}^2\Big).
\end{split}
\end{equation*}
Inserting the above two equalities in \eqref{eq-1-3}, we obtain
\begin{equation*}
\begin{split}
&\frac{1}{2\gamma \rho}\big\|\lambda^k - \lambda^*\big\|^2 + \frac{1}{2}\sum_{i=1}^N \big\|x_i^{k}-x_i^*\big\|_{\rho A_i^TA_i+P_i}^2
\\
&\geq \alpha \sum_{i=1}^N \big\|x_i^{k+1}-x_i^*\big\|^2 + \frac{1}{2\gamma \rho}\big\|\lambda^{k+1} - \lambda^*\big\|^2+\frac{1}{2}\sum_{i=1}^N \big\|x_i^{k+1}-x_i^*\big\|_{\rho A_i^TA_i+P_i}^2 + H_k
\end{split}
\end{equation*} 
where $H_k = \frac{1}{2}\sum_{i=1}^N \big\|x_i^{k+1}-x_i^k\big\|_{\rho A_i^TA_i+P_i}^2 + \frac{2-\gamma}{2\gamma^2\rho}\big\|\lambda^{k+1} - \lambda^{k}\big\|^2 + \frac{1}{\gamma} \Big\langle \lambda^k - \lambda^{k+1}, A\big(x^k - x^{k+1}\big)\Big\rangle$.
This is the desired inequality.
\end{proof}
We notice that the term $\alpha \sum_{i=1}^N \left\|x_i^{k+1} - x_i^*\right\|^2$ in the inequality \eqref{eq-1-20} is a gain for the strongly convexity of the costs. For the linear convergence, we make it to bound the distance $\left\| \lambda^k - \lambda^*\right\|^2$ between the multipliers $\lambda^k$ and $\lambda^*$ in the following lemma.

\begin{lem}\label{lem-1}
Suppose Assumption \ref{ass-1}-\ref{ass-5} hold.
\sloppy
Let $L = \max_{1\leq i \leq N}L_i^2$ and $D = \sum_{i=1}^N \left\| A_i^T\right\|^2$.
Then it holds that 
\begin{equation}\label{eq-3-3}
\begin{split}
&c_A^2\left\|\lambda^k-\lambda^* \right \|^2 
\\
& \leq   2L\sum_{i=1}^N\big\|x_i^{k+1}-x_i^*\big\|^2 + 4\rho^2D\bigg\|\sum_{j=1}^N A_j \big(x_j^k-x_j^*\big) \bigg\|^2 
\\
& \quad + 4\sum_{i=1}^N\bigg\|\Big(\rho A_i^TA_i + P_i \Big)\big(x_i^{k+1} - x_i^k \big) \bigg\|^2.
\end{split}
\end{equation}
\end{lem}
\begin{proof}
Using Lemma \ref{lem-2} and \eqref{eq-1-12}, we get
\begin{equation*}
\begin{split}
c_A^2\big\|\lambda^k - \lambda^*\big\|^2 & \leq \sum_{i=1}^N \big\|A_i^T (\lambda^k-\lambda^*)\big\|^2
\\
& \leq \sum_{i=1}^N \bigg\|\nabla f_i\big(x_i^{k+1}\big)-\nabla f_i\left(x_i^{*}\right)+ \rho A_i^T \Big( A_i \big(x_i^{k+1}-x_i^{*} \big) + \sum_{j \neq i}A_j\big( x_j^k - x_j^*\big)\Big)
\\
&\qquad + P_i \big(x_i^{k+1}-x_i^{k} \big)\bigg\|^2.
\end{split}
\end{equation*}
We apply the Cauchy-Schwartz inequality to deduce
\begin{equation*}
\begin{split}
&c_A^2\left\|\lambda^k-\lambda^* \right \|^2 
\\
& \leq 2\sum_{i=1}^N \Big\|\nabla f_i\big(x_i^{k+1}\big)-\nabla f_i\left(x_i^{*}\right)\Big\|^2
\\
&\qquad + 2\sum_{i=1}^N\bigg\|\rho A_i^T \Big( A_i \big(x_i^{k+1}-x_i^{*} \big) + \sum_{j \neq i}A_j\big( x_j^k - x_j^*\big)\Big)+ P_i \big(x_i^{k+1}-x_i^{k} \big)\bigg\|^2
\\
& = 2\sum_{i=1}^N \Big\|\nabla f_i\big(x_i^{k+1}\big)-\nabla f_i\left(x_i^{*}\right)\Big\|^2
\\
&\qquad + 2\sum_{i=1}^N\bigg\|\rho A_i^T\Big(\sum_{j=1}^N A_j \big(x_j^k-x_j^*\big) \Big) + \Big(\rho A_i^TA_i + P_i \Big)\big(x_i^{k+1} - x_i^k \big) \bigg\|^2,
\end{split}
\end{equation*}
where the latter equality holds by a simple rearrangement.
Now we apply Assumption \ref{ass-3} and the Cauch-Schwartz inequality to obtain
\begin{equation*}
\begin{split}
&c_A^2\left\|\lambda^k-\lambda^* \right \|^2
\\
& \leq 2\sum_{i=1}^N L_i^2\big\|x_i^{k+1}-x_i^*\big\|^2 
\\
&\qquad + 4\rho^2\sum_{i=1}^N\bigg\| A_i^T\Big(\sum_{j=1}^N A_j \big(x_j^k-x_j^*\big) \Big)\bigg\|^2 + 4\sum_{i=1}^N\bigg\|\Big(\rho A_i^TA_i + P_i \Big)\big(x_i^{k+1} - x_i^k \big) \bigg\|^2.
\end{split}
\end{equation*}
Applying the Cauchy-Schwartz inequality once more, we obtain
\begin{equation*}
\begin{split}
& c_A^2\left\|\lambda^k-\lambda^* \right \|^2 
\\
& \leq 2\left(\max_{1\leq i \leq N} L_i^2\right) \sum_{i=1}^N \big\|x_i^{k+1}-x_i^*\big\|^2 
\\
& \qquad + 4\rho^2\sum_{i=1}^N\big\| A_i^T\big\|^2\bigg\|\sum_{j=1}^N A_j \big(x_j^k-x_j^*\big) \bigg\|^2 + 4\sum_{i=1}^N\bigg\|\Big(\rho A_i^TA_i + P_i \Big)\big(x_i^{k+1} - x_i^k \big) \bigg\|^2
\\
& = 2L\sum_{i=1}^N\big\|x_i^{k+1}-x_i^*\big\|^2 
   + 4\rho^2D\bigg\|\sum_{j=1}^N A_j \big(x_j^k-x_j^*\big) \bigg\|^2 
\\   
& \qquad + 4\sum_{i=1}^N\bigg\|\Big(\rho A_i^TA_i + P_i \Big)\big(x_i^{k+1} - x_i^k \big) \bigg\|^2.
\end{split}
\end{equation*}
This proves the lemma.
\end{proof}
We derive the following result using Lemma \ref{lem-3-2} and Lemma \ref{lem-1}.
\begin{prop}\label{prop-3-4}
Suppose Assumptions \ref{ass-1}-\ref{ass-5} hold.
For $s>0$, we have the following inequality 
\begin{equation}\label{eq-3-5}
\begin{split}
&\left(\frac{1}{2\gamma \rho} - s c_A^2 \right)\big\|\lambda^k - \lambda^*\big\|^2   +  \frac{1}{2}\sum_{i=1}^N \big\|x_i^k - x_i^*\big\|_{(\rho+4Ns\rho^2 D) A_i^TA_i + P_i}^2 
\\
&\geq  \frac{1}{2\gamma \rho}\big\|\lambda^{k+1} - \lambda^*\big\|^2
  +  \frac{1}{2} \sum_{i=1}^N \big\|x_i^{k+1} - x_i^*\big\|_{\rho A_i^TA_i + P_i + 2(\alpha -2 Ls)I_{n_i}}^2 + G_k,
\end{split}
\end{equation}
where $H_k := \frac{1}{2} \sum_{i=1}^N \big\|x_i^{k+1}-x_i^k\big\|_{\rho A_i^TA_i+P_i}^2 + \frac{2-\gamma}{2\gamma^2\rho}\big\|\lambda^{k+1} - \lambda^{k}\big\|^2 + \frac{1}{\gamma} \Big\langle \lambda^k - \lambda^{k+1}, A\big(x^k - x^{k+1}\big)\Big\rangle$ and $G_k := H_k -4s\sum_{i=1}^N\bigg\|\Big(\rho A_i^TA_i + P_i \Big)\big(x_i^{k+1} - x_i^k \big) \bigg\|^2$.
\end{prop}

\begin{proof}
We recall from Lemma \ref{lem-3-2} the following inequality
\begin{equation}\label{eq-3-1}
\begin{split}
&\frac{1}{2\gamma \rho}\big\|\lambda^k - \lambda^*\big\|^2 + \frac{1}{2}\sum_{i=1}^N \big\|x_i^{k}-x_i^*\big\|_{\rho A_i^TA_i+P_i}^2
\\
&\geq \alpha \sum_{i=1}^N \big\|x_i^{k+1}-x_i^*\big\|^2 + \frac{1}{2\gamma \rho}\big\|\lambda^{k+1} - \lambda^*\big\|^2+\frac{1}{2}\sum_{i=1}^N \big\|x_i^{k+1}-x_i^*\big\|_{\rho A_i^TA_i+P_i}^2
\\
&+\frac{1}{2}\sum_{i=1}^N \big\|x_i^{k+1}-x_i^k\big\|_{\rho A_i^TA_i+P_i}^2 + \frac{2-\gamma}{2\gamma^2\rho}\big\|\lambda^{k+1} - \lambda^{k}\big\|^2 + \frac{1}{\gamma} \Big\langle \lambda^k - \lambda^{k+1}, A\big(x^k - x^{k+1}\big)\Big\rangle.
\end{split}
\end{equation} 
Next we recall the inequality of Lemma \ref{lem-1} given as
\begin{equation}\label{eq-1-13}
\begin{split}
&c_A^2\left\|\lambda^k-\lambda^* \right \|^2 
\\
& \leq   2L\sum_{i=1}^N\big\|x_i^{k+1}-x_i^*\big\|^2 
\\   
&\quad + 4\rho^2D\bigg\|\sum_{j=1}^N A_j \big(x_j^k-x_j^*\big) \bigg\|^2 + 4\sum_{i=1}^N\bigg\|\Big(\rho A_i^TA_i + P_i \Big)\big(x_i^{k+1} - x_i^k \big) \bigg\|^2.
\end{split}
\end{equation}
Summing \eqref{eq-1-13} multiplied by $s>0$ and \eqref{eq-3-1}, we get
\begin{equation}\label{eq-3-2}
\begin{split}
&\left(\frac{1}{2\gamma \rho} - s c_A^2 \right)\big\|\lambda^k - \lambda^*\big\|^2 +  \frac{1}{2}\sum_{i=1}^N \big\|x_i^k - x_i^*\big\|_{\rho A_i^TA_i + P_i}^2  + 4s\rho^2D\bigg\|\sum_{j=1}^N A_j \big(x_j^k-x_j^*\big) \bigg\|^2 
\\
&\geq (\alpha - 2Ls) \sum_{i=1}^N \big\|x_i^{k+1}-x_i^*\big\|^2 +\frac{1}{2\gamma \rho}\big\|\lambda^{k+1} - \lambda^*\big\|^2
  +  \frac{1}{2} \sum_{i=1}^N \big\|x_i^{k+1} - x_i^*\big\|_{\rho A_i^TA_i + P_i}^2
\\
& \quad + \frac{1}{2} \sum_{i=1}^N \big\|x_i^{k+1}-x_i^k\big\|_{\rho A_i^TA_i+P_i}^2 -4s\sum_{i=1}^N\bigg\|\Big(\rho A_i^TA_i + P_i \Big)\big(x_i^{k+1} - x_i^k \big) \bigg\|^2
\\
& \quad + \frac{2-\gamma}{2\gamma^2\rho}\big\|\lambda^{k+1} - \lambda^{k}\big\|^2 
 + \frac{1}{\gamma} \Big\langle \lambda^k - \lambda^{k+1}, A\big(x^k - x^{k+1}\big)\Big\rangle.
\end{split}
\end{equation}
We use the Cauchy-Schwartz inequality to find 
\begin{equation*}
4s\rho^2D\bigg\|\sum_{j=1}^N A_j \big(x_j^k-x_j^*\big) \bigg\|^2  \leq 4N s\rho^2D \sum_{j=1}^N \| A_j \big(x_j^k-x_j^*\big) \|^2. 
\end{equation*}
Using this inequality in \eqref{eq-3-2} and rearranging, we get
\begin{equation*}
\begin{split}
&\left(\frac{1}{2\gamma \rho} - s c_A^2 \right)\big\|\lambda^k - \lambda^*\big\|^2   +  \frac{1}{2}\sum_{i=1}^N \big\|x_i^k - x_i^*\big\|_{(\rho+4Ns\rho^2 D) A_i^TA_i + P_i}^2 
\\
&\geq  \frac{1}{2\gamma \rho}\big\|\lambda^{k+1} - \lambda^*\big\|^2
  +  \frac{1}{2} \sum_{i=1}^N \big\|x_i^{k+1} - x_i^*\big\|_{\rho A_i^TA_i + P_i + 2(\alpha -2 Ls)I_{n_i}}^2
\\
& \quad + \frac{1}{2} \sum_{i=1}^N \big\|x_i^{k+1}-x_i^k\big\|_{\rho A_i^TA_i+P_i}^2 -4s\sum_{i=1}^N\bigg\|\Big(\rho A_i^TA_i + P_i \Big)\big(x_i^{k+1} - x_i^k \big) \bigg\|^2
\\
& \quad + \frac{2-\gamma}{2\gamma^2\rho}\big\|\lambda^{k+1} - \lambda^{k}\big\|^2 
 + \frac{1}{\gamma} \Big\langle \lambda^k - \lambda^{k+1}, A\big(x^k - x^{k+1}\big)\Big\rangle,
\end{split}
\end{equation*}
which is the desired inequality. 
\end{proof}

To show the main convergence result Proposition \ref{thm-main}, we need the following two lemmas.
\begin{lem}\label{lem-3-5}If $s >0$ satisfies
\begin{equation}\label{eq-3-6}
s \Big[ \rho^2 D \|A_i \|^2 + \frac{L}{N}\Big] < \frac{\alpha}{2N},
\end{equation}  
there exists $\mu_s \in (0,1)$ such that
\begin{equation}\label{eq-3-7}
\|x\|^2_{(\rho +4 Ns \rho^2 D)A_i^T A_i + P_i} \leq \mu_s \|x\|^2_{\rho A_i^T A_i + P_i +2(\alpha -2Ls) I_{n_i}} \quad \forall~x \in \mathbb{R}^{n_i}.
\end{equation}
\end{lem}

\begin{proof}
If we choose $s>0$ that satisfies \eqref{eq-3-6}, we have 
    \begin{equation*}
        (\rho+4Ns\rho^2D)A_i^TA_i + P_i \prec \rho A_i^TA_i + P_i + 2(\alpha-2Ls)I_{n_i}.
    \end{equation*}
    Thus, we obtain the final result for $\mu_s \in (0,1)$, which is sufficiently close to $1$.
\end{proof}

\begin{lem}\label{lem-3-6} If one chooses some $\xi_i > 0$ such that $0 < \gamma < 2$, $\rho > 0$, and $s>0$ satisfies the following conditions:
\begin{equation*}
\left\{ \begin{array}{ll}
\rho A_i^TA_i + P_i -8s (\rho A_i^TA_i + P_i)^T(\rho A_i^TA_i + P_i) - \frac{\rho}{\xi_i}A_i^TA_i \succ 0,
 \\
\sum_{i=1}^N \xi_i < 2- \gamma,
\end{array}\right.
\end{equation*}
 then we have
\begin{equation*}
\begin{split}
&\frac{1}{2} \sum_{i=1}^N \big\|x_i^{k+1}-x_i^k\big\|_{\rho A_i^TA_i+P_i}^2 -4s\sum_{i=1}^N\bigg\|\Big(\rho A_i^TA_i + P_i \Big)\big(x_i^{k+1} - x_i^k \big) \bigg\|^2
\\
& \quad + \frac{2-\gamma}{2\gamma^2\rho}\big\|\lambda^{k+1} - \lambda^{k}\big\|^2 
 + \frac{1}{\gamma} \Big\langle \lambda^k - \lambda^{k+1}, A\big(x^k - x^{k+1}\big)\Big\rangle\geq 0.
\end{split}
\end{equation*}
\end{lem}
\begin{proof}
Using Young's inequality, we have
\begin{equation}\label{eq-3-4}
\begin{split}
\frac{2}{\gamma} \Big\langle \lambda^k - \lambda^{k+1}, A\big(x^k - x^{k+1}\big)\Big\rangle &= \sum_{i=1}^N \frac{2}{\gamma} (\lambda^k - \lambda^{k+1})^T A_i(x_i^k - x_i^{k+1})
\\
&\geq -\sum_{i=1}^N \bigg(\frac{\xi_i}{\rho \gamma^2}\| \lambda^{k+1} - \lambda^{k} \|^2 + \frac{\rho}{\xi_i}\|A_i(x_i^k - x_i^{k+1})\|^2 \bigg)
\\
& = -\frac{\sum_{i=1}^N\xi_i}{\rho \gamma^2}\| \lambda^{k+1} - \lambda^{k} \|^2 -\sum_{i=1}^N \frac{\rho}{\xi_i}\|A_i(x_i^k - x_i^{k+1})\|^2,
\end{split}
\end{equation}
for $\xi_j>0~(1\leq j\leq N)$.
Adding $\sum_{i=1}^N \big\|x_i^{k+1}-x_i^k\big\|_{\rho A_i^TA_i+P_i}^2$, $-8s\sum_{i=1}^N\bigg\|\Big(\rho A_i^TA_i + P_i \Big)\big(x_i^{k+1} - x_i^k \big) \bigg\|^2$, and $\frac{2-\gamma}{\gamma^2\rho}\big\|\lambda^{k+1} - \lambda^{k}\big\|^2$ to both sides of \eqref{eq-3-4}, we obtain
\begin{equation*}
\begin{split}
& \sum_{i=1}^N \big\|x_i^{k+1}-x_i^k\big\|_{\rho A_i^TA_i+P_i}^2 -8s\sum_{i=1}^N\bigg\|\Big(\rho A_i^TA_i + P_i \Big)\big(x_i^{k+1} - x_i^k \big) \bigg\|^2
\\
& \quad + \frac{2-\gamma}{\gamma^2\rho}\big\|\lambda^{k+1} - \lambda^{k}\big\|^2 
 + \frac{2}{\gamma} \Big\langle \lambda^k - \lambda^{k+1}, A\big(x^k - x^{k+1}\big)\Big\rangle
\\
&\geq \sum_{i=1}^N \|x_i^{k+1}-x_i^k\|^2_{\rho A_i^TA_i+P_i -8s (\rho A_i^TA_i+P_i)^T(\rho A_i^TA_i+P_i) - \frac{\rho}{\xi_i}A_i^TA_i}
\\
&+ \bigg(\frac{2-\gamma-\sum_{i=1}^N \xi_i}{\gamma^2\rho}\bigg)\|\lambda^{k+1} - \lambda^{k}\|^2
\end{split}
\end{equation*}
for $\xi_j > 0 ~ (j=1, \cdots, N)$.
The proof is done.
\end{proof} 
 
We are now ready to show our main convergence and rate results.
\begin{proof}[Proof of Theorem \ref{thm-main}]
By applying Lemma \ref{lem-3-5} and Lemma \ref{lem-3-6} to the inequality of Proposition \ref{prop-3-4}, we get
\begin{equation*} 
\begin{split}
&\left(\frac{1}{2\gamma \rho} - s c_A^2 \right)\big\|\lambda^k - \lambda^*\big\|^2    +  \frac{\mu_s}{2}\sum_{i=1}^N \big\|x_i^k - x_i^*\big\|_{\rho A_i^TA_i + P_i + 2(\alpha -2 Ls)I_{n_i}}^2    
\\
&\geq  \frac{1}{2\gamma \rho}\big\|\lambda^{k+1} - \lambda^*\big\|^2
  +  \frac{1}{2} \sum_{i=1}^N \big\|x_i^{k+1} - x_i^*\big\|_{\rho A_i^TA_i + P_i + 2(\alpha -2 Ls)I_{n_i}}^2. 
\end{split}
\end{equation*}
This gives the following inequality 
\begin{equation*} 
\begin{split}
& \sigma \bigg( \frac{1}{2\gamma \rho}  \big\|\lambda^k - \lambda^*\big\|^2    +  \frac{1}{2}\sum_{i=1}^N \big\|x_i^k - x_i^*\big\|_{\rho A_i^TA_i + P_i + 2(\alpha -2 Ls)I_{n_i}}^2  \bigg)  
\\
&\geq  \frac{1}{2\gamma \rho}\big\|\lambda^{k+1} - \lambda^*\big\|^2
  +  \frac{1}{2} \sum_{i=1}^N \big\|x_i^{k+1} - x_i^*\big\|_{\rho A_i^TA_i + P_i + 2(\alpha -2 Ls)I_{n_i}}^2,
\end{split}
\end{equation*}
where $\sigma = \max \left\{ {1-2\gamma\rho sc_A^2,\mu_s}\right\}$.
The proof is done.
\end{proof}

\section{Numerical experiments}\label{sec-4}
In this section, we present numerical tests of Algorithm \ref{algo} supporting the linear convergence property proved in Theorem \ref{thm-main}. 
We test the Algorithm \ref{algo} for the linearly const rained quadratic programming (LCQP) model \cite{A_faster} and the optimal source allocation problem \cite{Optimal_scaling}.
In Section~\ref{sec-4}, all numerical experiments were performed using Python~3.10.16 with NumPy~2.1.3, SciPy~1.10.1, and Matplotlib~3.9.2 on a workstation with an Intel Core i9-12900KS CPU (3.40\,GHz) and 32.0\,GB RAM.

\subsection{The LCQP Model}\label{subsec-4}

Here we consider the following LCQP model:
\begin{equation}\label{model-1}
\begin{split}
& \min \sum_{i=1}^N f_i(x_i)
\\
& \textrm{s.t.}~ \sum_{i=1}^N A_ix_i = c
\end{split}
\end{equation}
where $f_i:\mathbb{R}^{n_i}\rightarrow \mathbb{R}~\textrm{is given by}~f_i(x_i) = \frac{1}{2}x_i^TH_ix_i + x_i^Tq_i \textrm{ with} ~ H_i\in \mathbb{S}_+^{n_i}, ~q_i \in \mathbb{R}^{n_i},~ A_i \in \mathbb{R}^{m \times n_i}~ \textrm{and} ~c\in \mathbb{R}^m$. 
Since $H_i\in \mathbb{S}_+^{n_i}$, the quadratic function $f_i$ is strongly convex for each $i=1,\ldots,N$.
We take $n_1=\cdots=n_N:=n$ in our experiment. 
We conduct experiments with the damping parameter $\gamma \in \{0.1, 0.5, 1.5, 1.9\}$, and the penalty parameter $\rho$, chosen as $\{0.03, 1, 5, 10\}$ for $N=3$ and $\{10^{-5}, 0.1, 5, 10\}$ for $N=10$, in Algorithm \ref{algo}.
The matrices involved in this problem are randomly generated through the following procedure:
\begin{enumerate}
\item Each $A_i\in\mathbb{R}^{m \times n_i}(i=1,\ldots, N)$ is randomly generated with i.i.d. standard normal entries with satisfying the condition that $\big[A_1^T; \ldots; A_N^T\big]$ has full column rank.
\item The positive semi-definite matrices $P_i \in \mathbb{R}^{n_i \times n_i} (i=1,\ldots, N)$ are generated by sampling $B \in \mathbb{R}^{n_i \times n_i}$ with i.i.d. standard normal entries, and then symmetrizing as $P_i = \frac{1}{2}(B+B^T)$ with checking that the eigenvalues of $P_i$ are non-negative.
\item The positive definite matrices $H_i \in \mathbb{S}_+^{n_i} (i=1,\ldots,N)$ are generated by first sampling $B \in \mathbb{R}^{n_i \times n_i} $ with i.i.d. standard normal entries, and then symmetrizing as $H_i = \frac{1}{2}(B+B^T)$ with checking that the eigenvalues of $H_i$ are positive.
\item We first generate the optimal solution $x_i^* \in \mathbb{R}^{n_i}$ and $\lambda^* \in \mathbb{R}^{m}$ as follows: $x_i^* = randn(n_i,1)$ for $i = 1, \ldots, N$, and $\lambda^* = randn(m,1)$;
\item Next we define $q_i\in \mathbb{R}^{n_i}(i=1,\cdots,N)$ and $c \in \mathbb{R}^m$ as follows: $q_i = -H_i x_i^* + A_i^T \lambda^*$ and $c = \sum_{i=1}^N A_ix_i^*$ so that $\left({x_1^*}^T,\cdots,{x_N^*}^T, {\lambda^*}^T \right)^T$ satisfies the KKT conditions.
\end{enumerate}
We update $x_i^{k+1}$ in \eqref{eq-1-5} using the equation from the optimality given as follows:
\begin{equation*}
H_ix_i^{k+1} + q_i + \rho A_i^T\Bigg(A_ix_i^{k+1} + \sum_{j\neq i}A_jx_j^k -c- \frac{\lambda^k}{\rho} \Bigg) + P_i\big(x_i^{k+1} - x_i^{k}\big)=0.
\end{equation*}
We define the error metric at each $k$-th iteration for $u^k = \big(x_1^k, \ldots, x_N^k ; \lambda^k\big)$ as follows:
\begin{equation}\label{eq-4-1}
\text{dis}(u^k) := \max \Big\{\big\|x_1^k-x_1^*\big\|,\ldots, \big\|x_N^k-x_N^*\big\|,\big\| \lambda^k-\lambda^* \big\| \Big\}.
\end{equation}
The initial values $u^0 = (x_1^0, \ldots, x_N^0; \lambda^0)$ are set to zero.
The experiments were conducted for $N=3$ with $m=100$ and $n=40$, and for $N=10$ with $m=100$ and $n=60$.
The experimental results are plotted in Figures \ref{fig-1} and \ref{fig-2} for the $3$-block case, and in Figures \ref{fig-3} and \ref{fig-4} for the $10$-block case.  
We find that the value $\text{dis}(u^k)$ decreases to zero linearly as expected in Theorem \ref{thm-main}.

As shown in Figure \ref{fig-1}, for fixed values of the damping parameter $\gamma \in \{0.1, 0.5\}$, the algorithm exhibits the fastest convergence when $\rho = 1$.
This indicates that a smaller penalty parameter does not necessarily lead to faster convergence, emphasizing the importance of proper choice of the penalty parameter $\rho$ for the fast convergence.
Figure \ref{fig-2} shows the effect of the damping parameter $\gamma$ under fixed $\rho$. 
The convergence rate improves as $\gamma$ increases when $\rho = 0.03$, whereas for $\rho = 5$, the algorithm exhibits similar convergence rates across all tested values of $\gamma$, indicating limited sensitivity to the damping parameter.

\begin{figure}[htbp]
    \centering
    \subfloat[$\rho = \{0.03, 1, 5, 10\}$, $\gamma = 0.1$]{{\includegraphics[width=6cm]{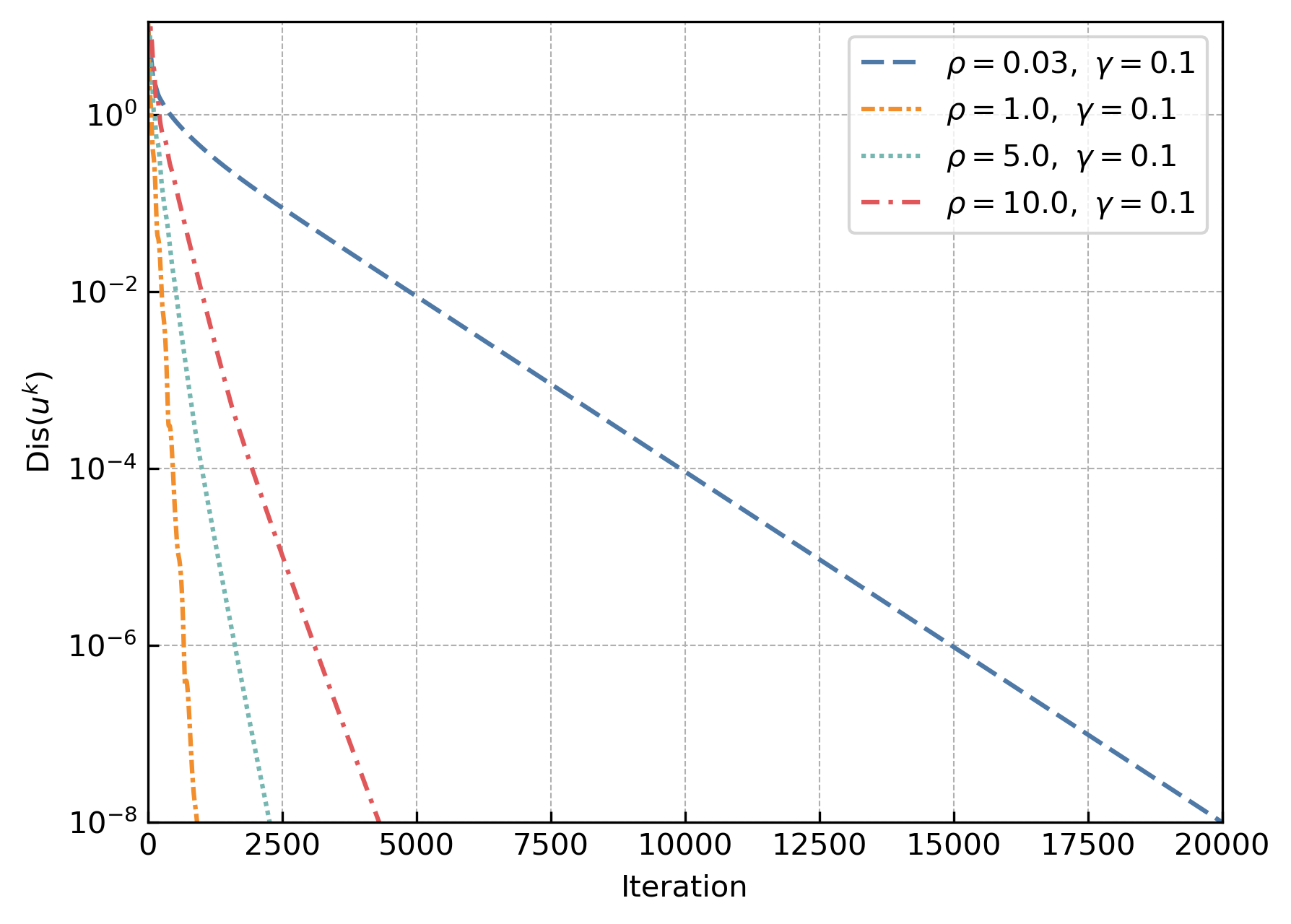} }}
    \subfloat[$\rho = \{0.03, 1, 5, 10\}$, $\gamma = 0.5$]{{\includegraphics[width=6cm]{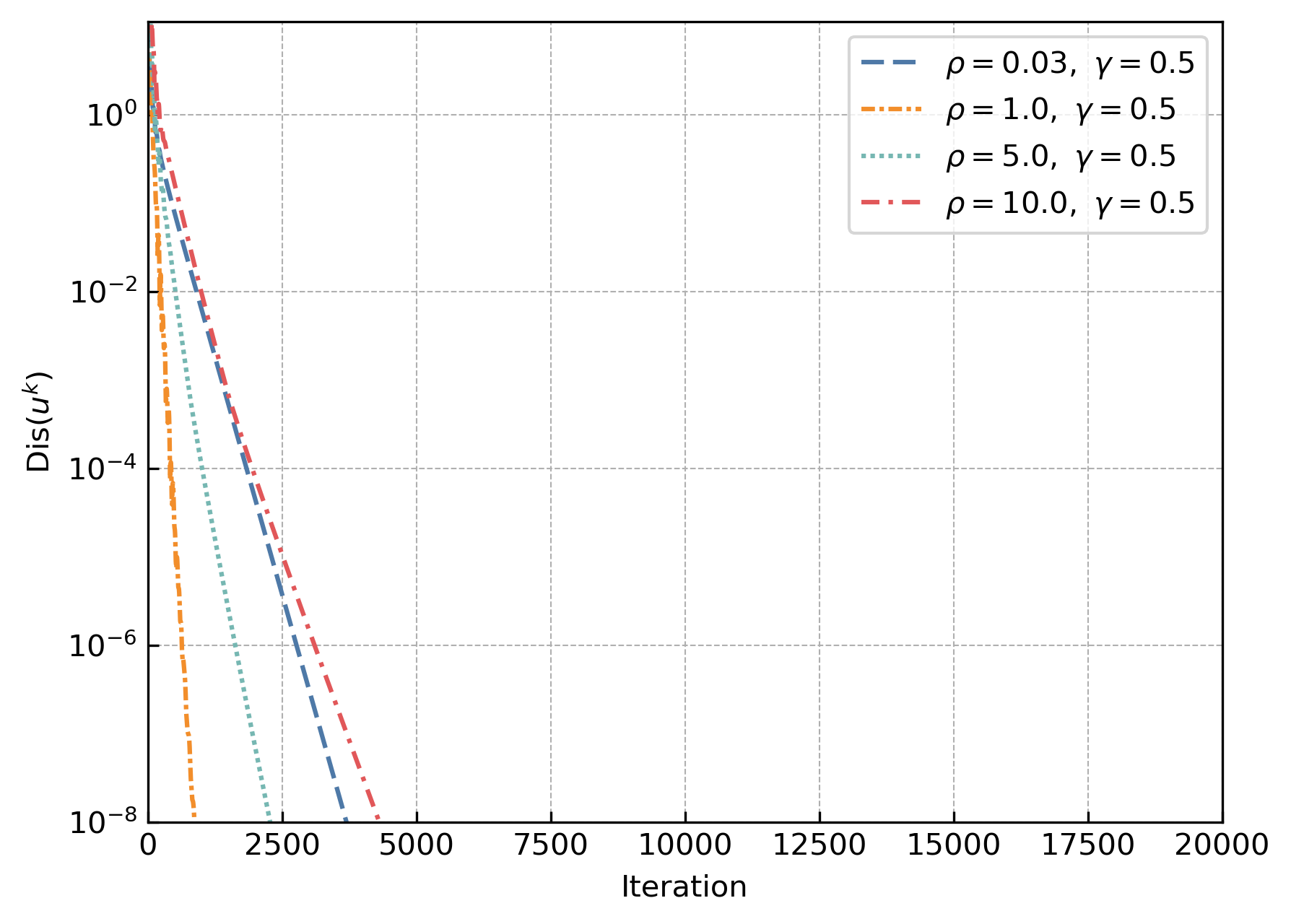} }}
    \\
    \subfloat[$\rho = \{0.03, 1, 5, 10\}$, $\gamma = 1.5$]{{\includegraphics[width=6cm]{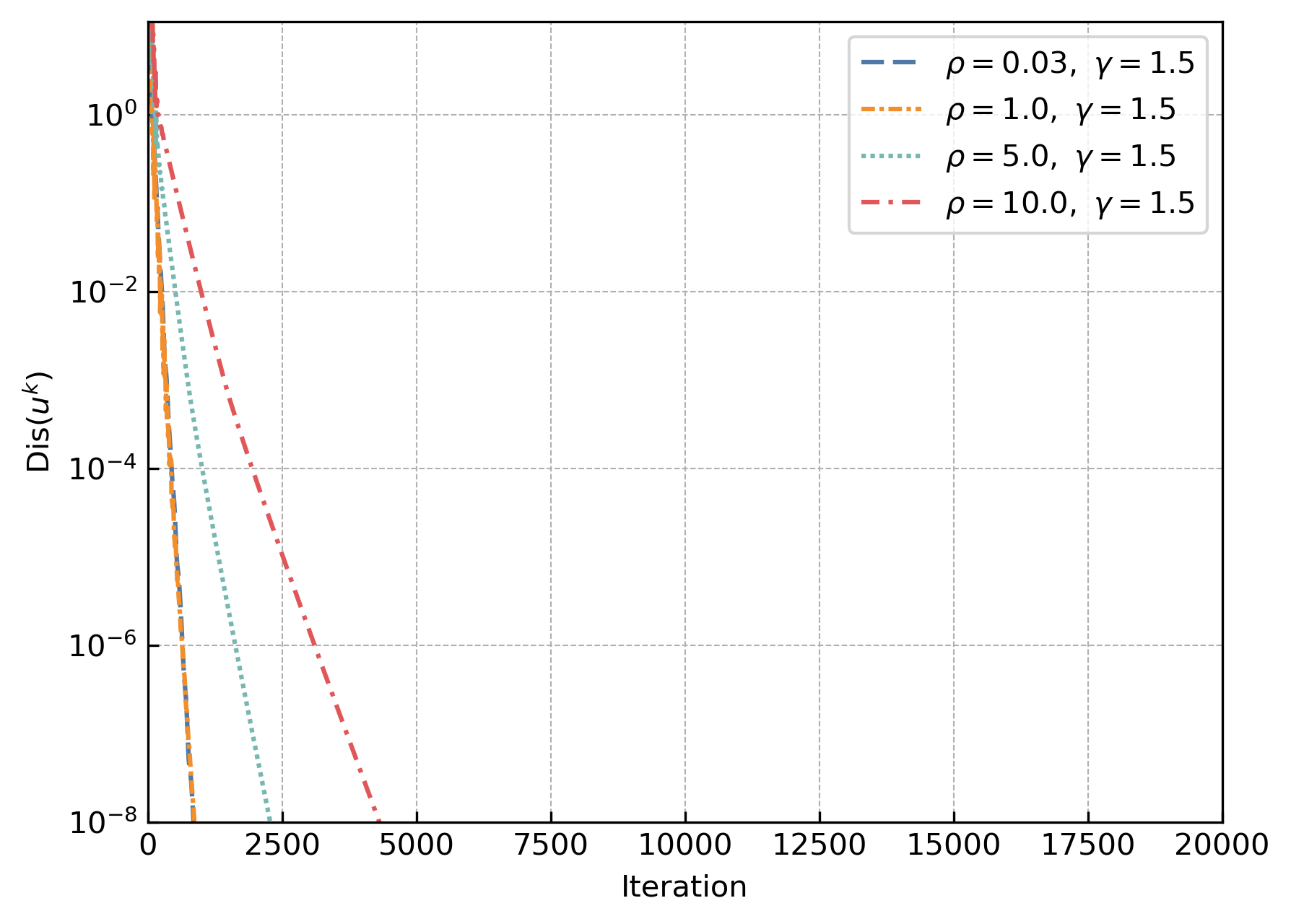} }}
    \subfloat[$\rho = \{0.03, 1, 5, 10\}$, $\gamma = 1.9$]{{\includegraphics[width=6cm]{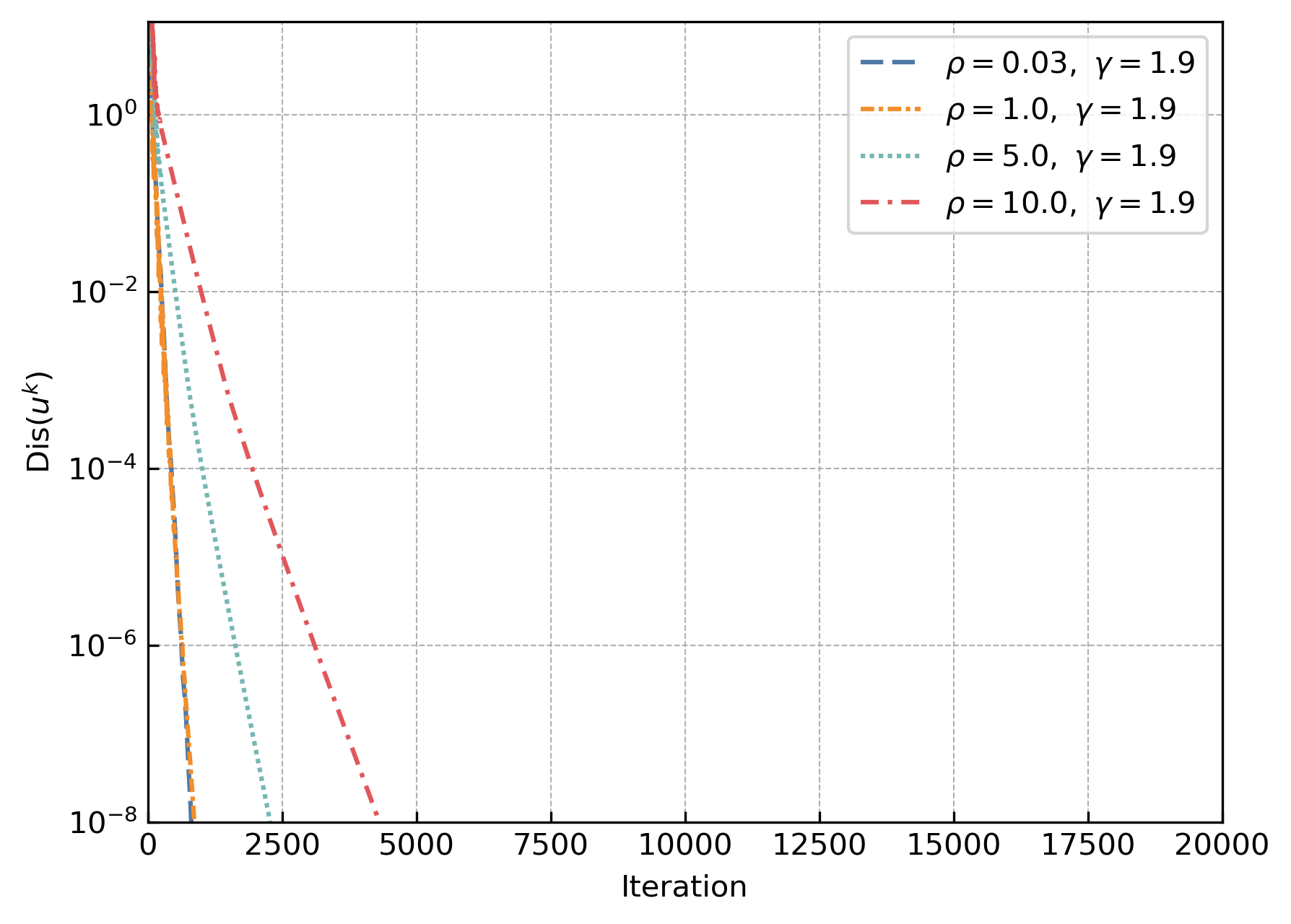} }}
    \caption{Experimental results on LCQP model for $N=3$ ($m=100$, $n=40$), with fixed $\gamma$ and varying $\rho$.}\label{fig-1}
\end{figure}

\begin{figure}[htbp]
    \centering
    \subfloat[$\rho = 0.03$, $\gamma = \{0.1, 0.5, 1.5, 1.9\}$]{{\includegraphics[width=6cm]{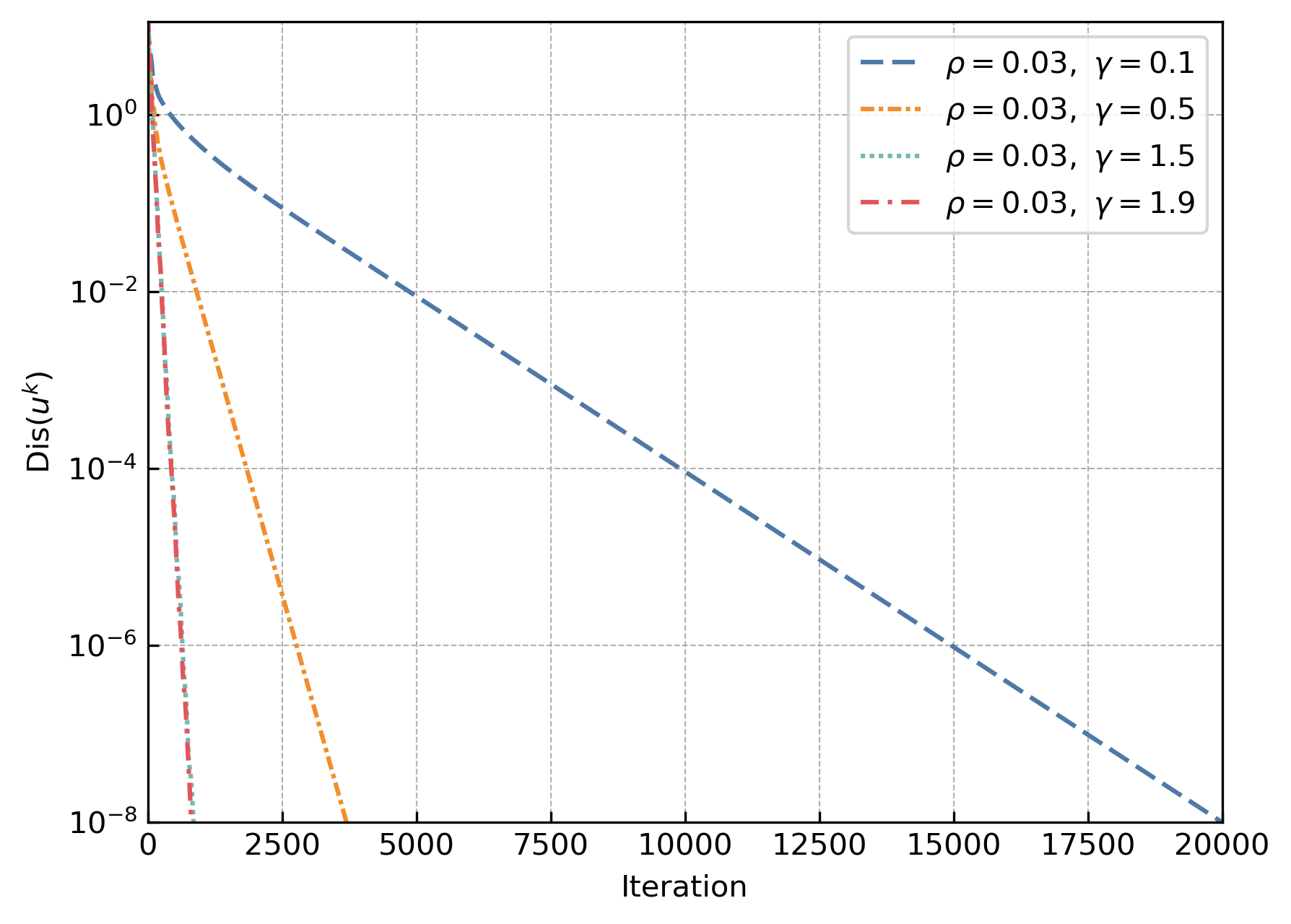} }}
    \subfloat[$\rho = 5$, $\gamma = \{0.1, 0.5, 1.5, 1.9\}$]{{\includegraphics[width=6cm]{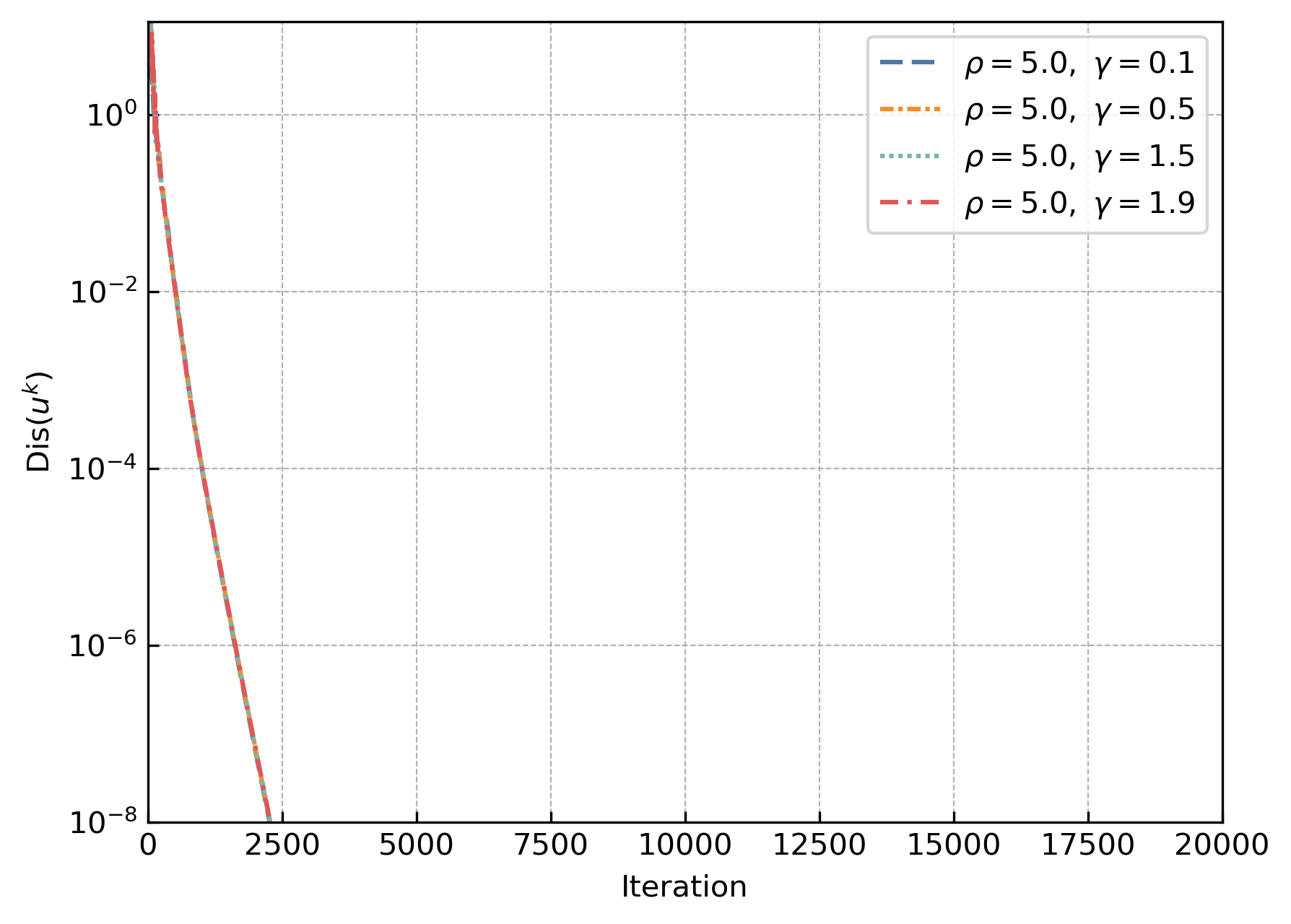} }}
    \caption{Experimental results on LCQP model for $N=3$ ($m=100$, $n=40$), with fixed $\rho$ and varying $\gamma$.}\label{fig-2}
\end{figure}

\begin{figure}[htbp]
    \centering
    \subfloat[$\rho = \{10^{-5}, 0.1, 5, 10\}$, $\gamma = 0.1$]{{\includegraphics[width=6cm]{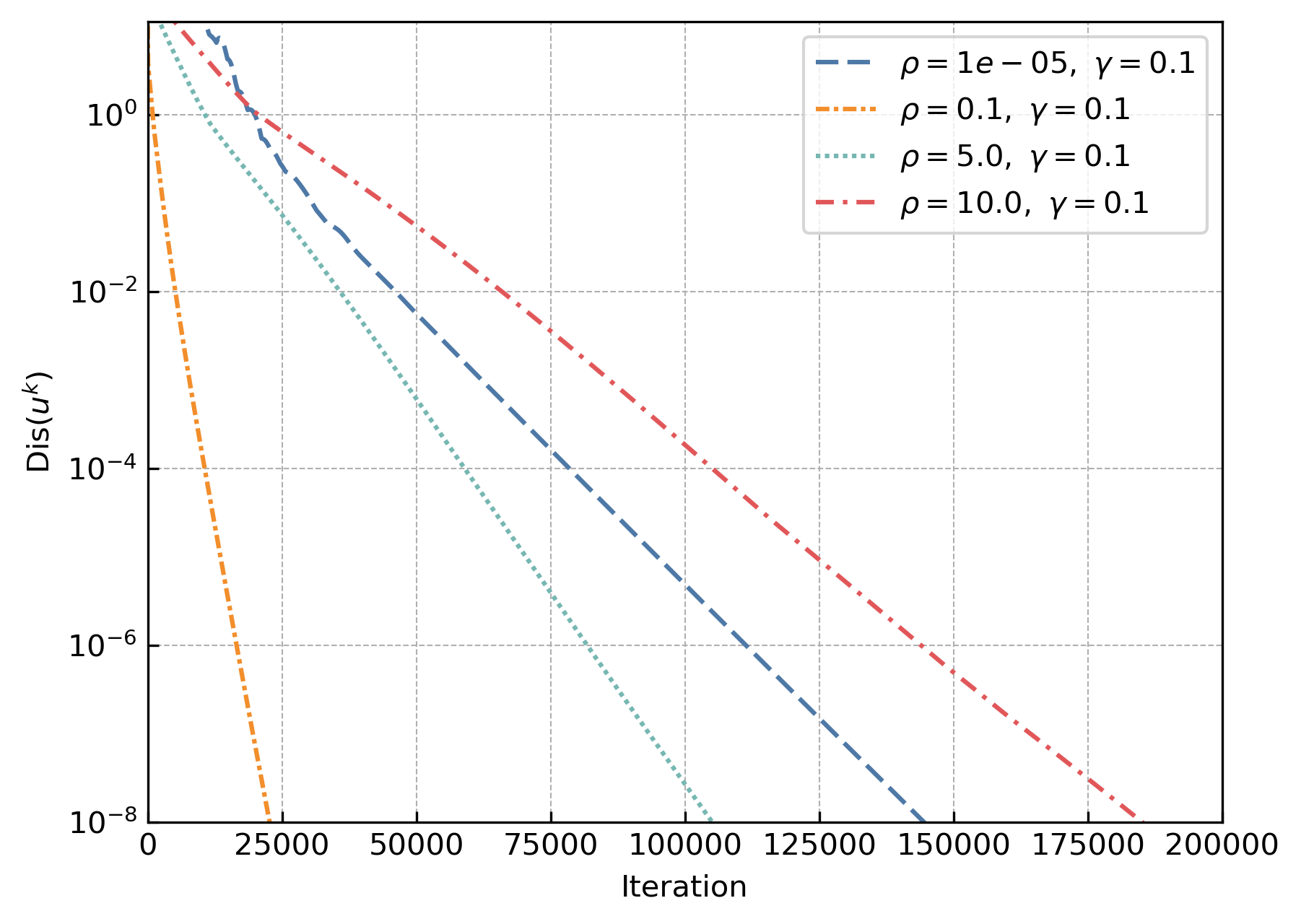} }}
    \subfloat[$\rho = \{10^{-5}, 0.1, 5, 10\}$, $\gamma = 0.5$]{{\includegraphics[width=6cm]{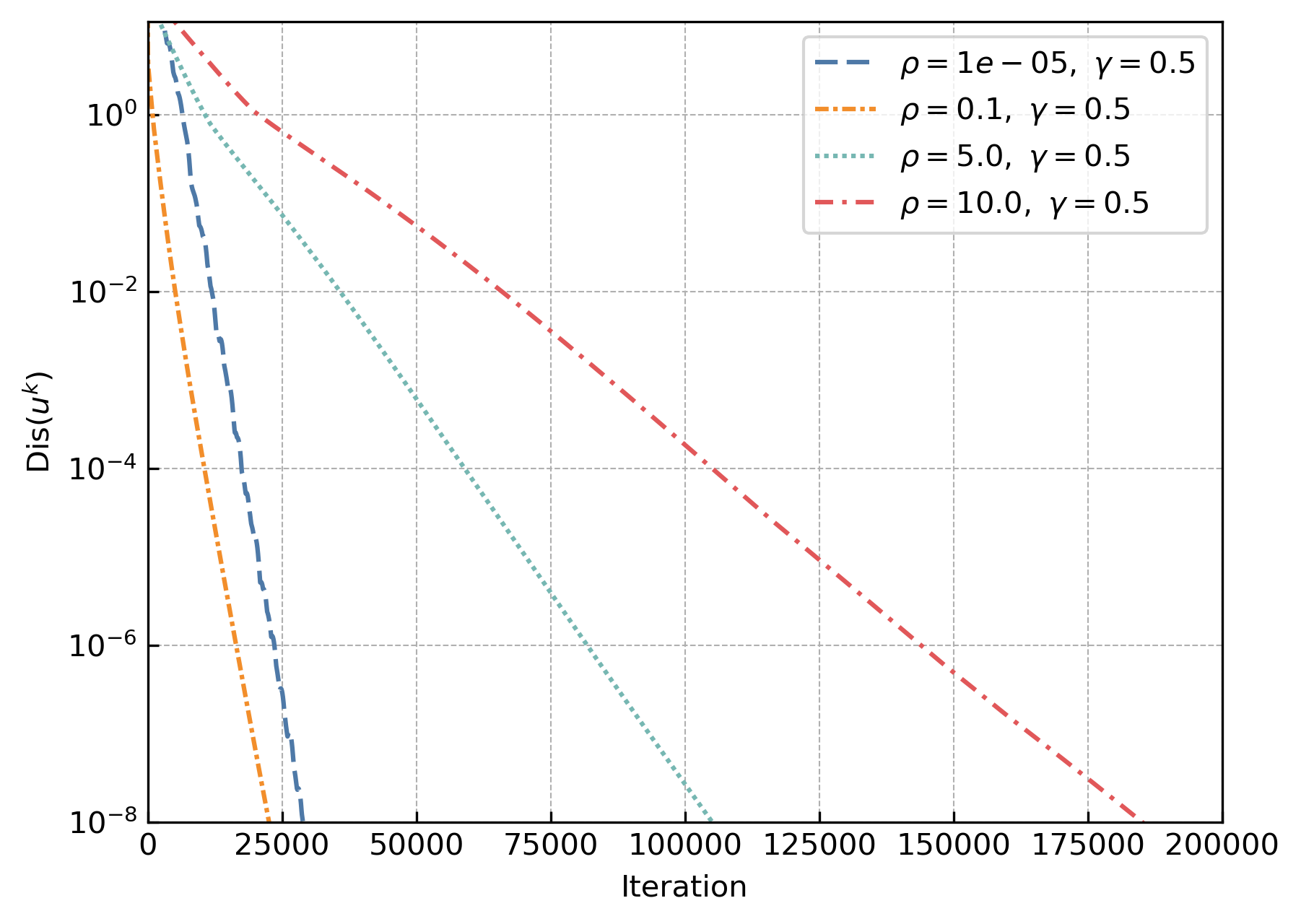} }}
    \\
    \subfloat[$\rho = \{10^{-5}, 0.1, 5, 10\}$, $\gamma = 1.5$]{{\includegraphics[width=6cm]{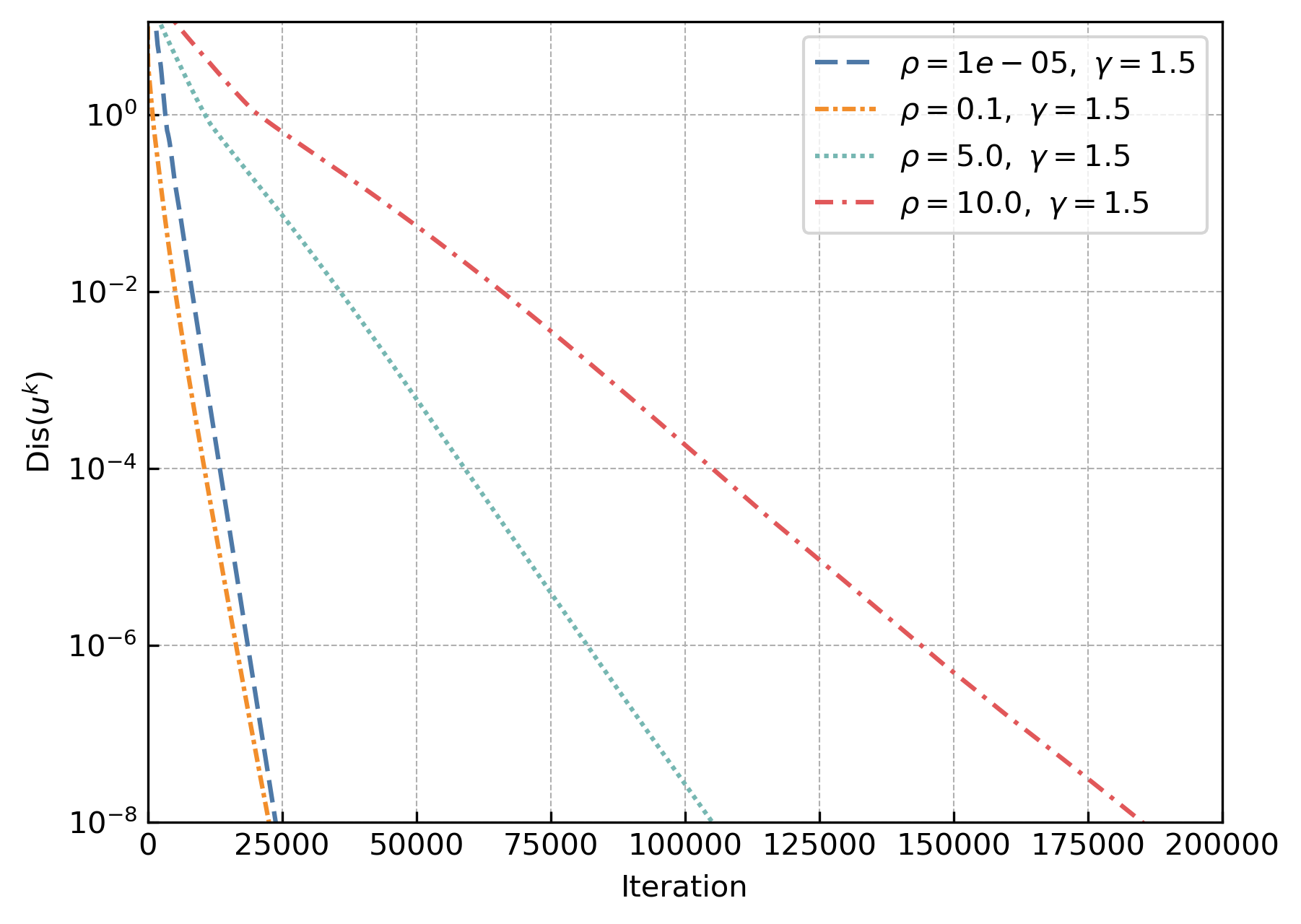} }}
    \subfloat[$\rho = \{10^{-5}, 0.1, 5, 10\}$, $\gamma = 1.9$]{{\includegraphics[width=6cm]{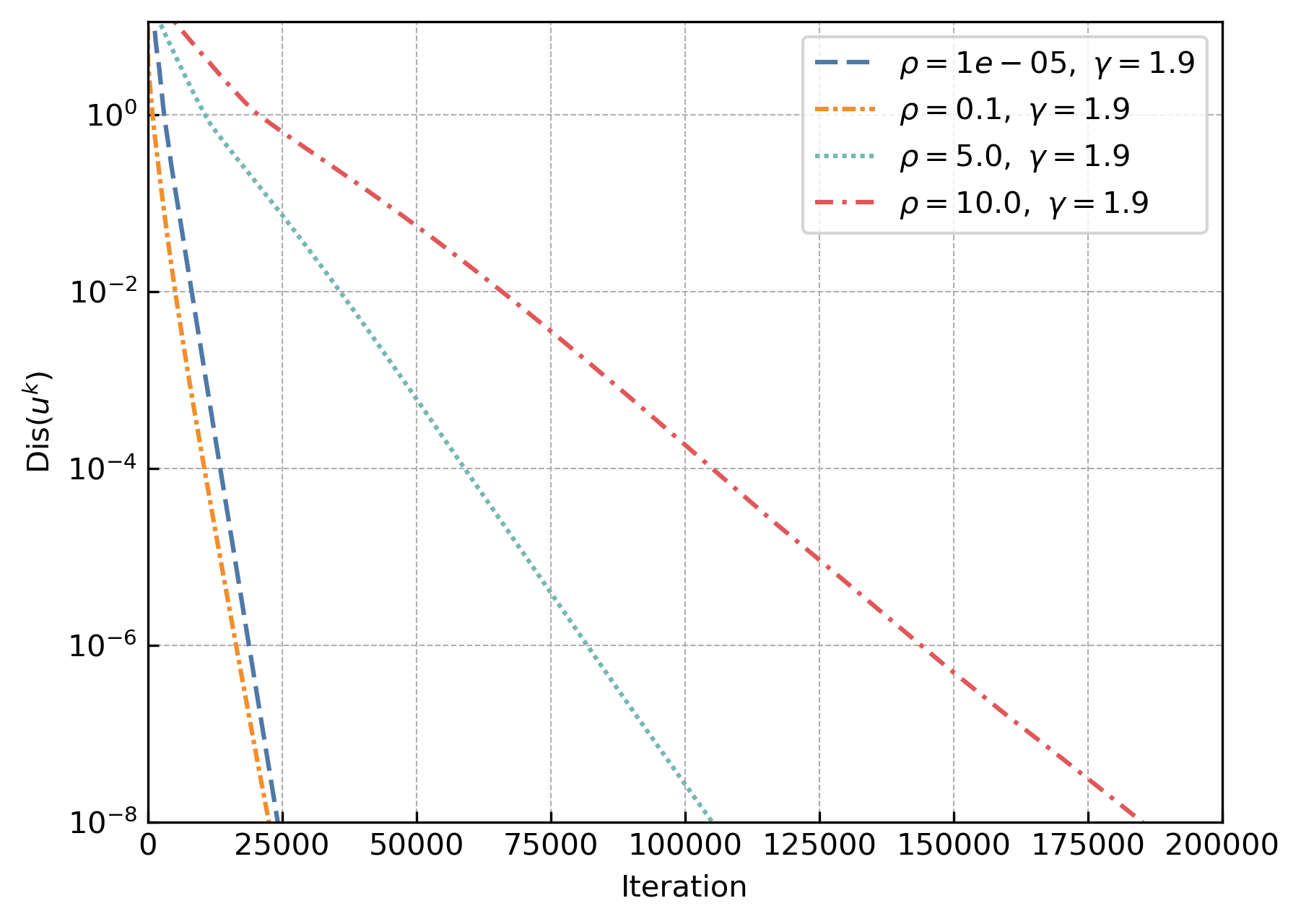} }}
    \caption{Experimental results on LCQP model for $N=10$ ($m=100$, $n=60$), with fixed $\gamma$ and varying $\rho$.}\label{fig-3}
\end{figure}

\begin{figure}[htbp]
    \centering
    \subfloat[$\rho = 10^{-5}$, $\gamma = \{0.1, 0.5, 1.5, 1.9\}$]{{\includegraphics[width=6cm]{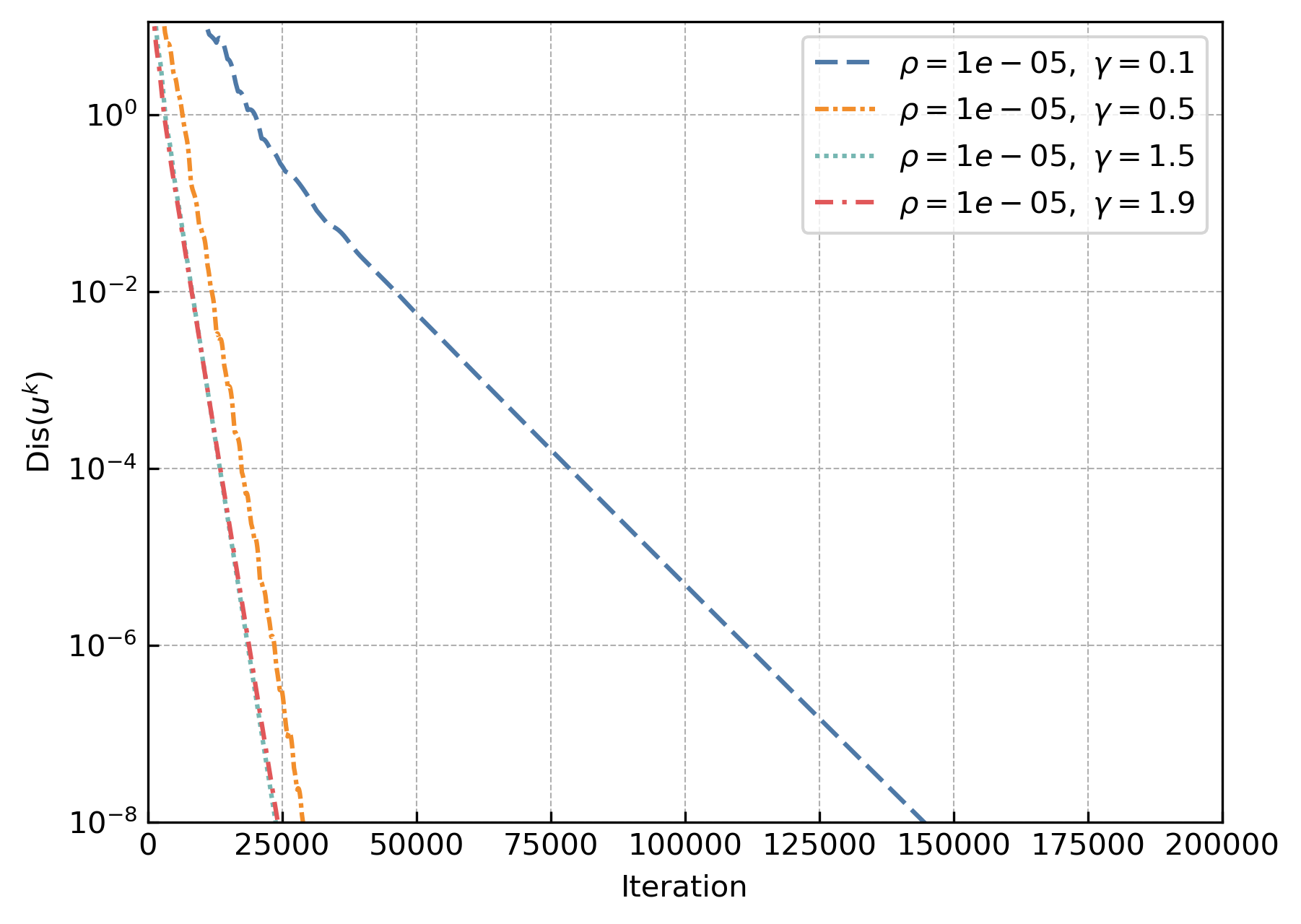} }}
    \subfloat[$\rho = 5$, $\gamma = \{0.1, 0.5, 1.5, 1.9\}$]{{\includegraphics[width=6cm]{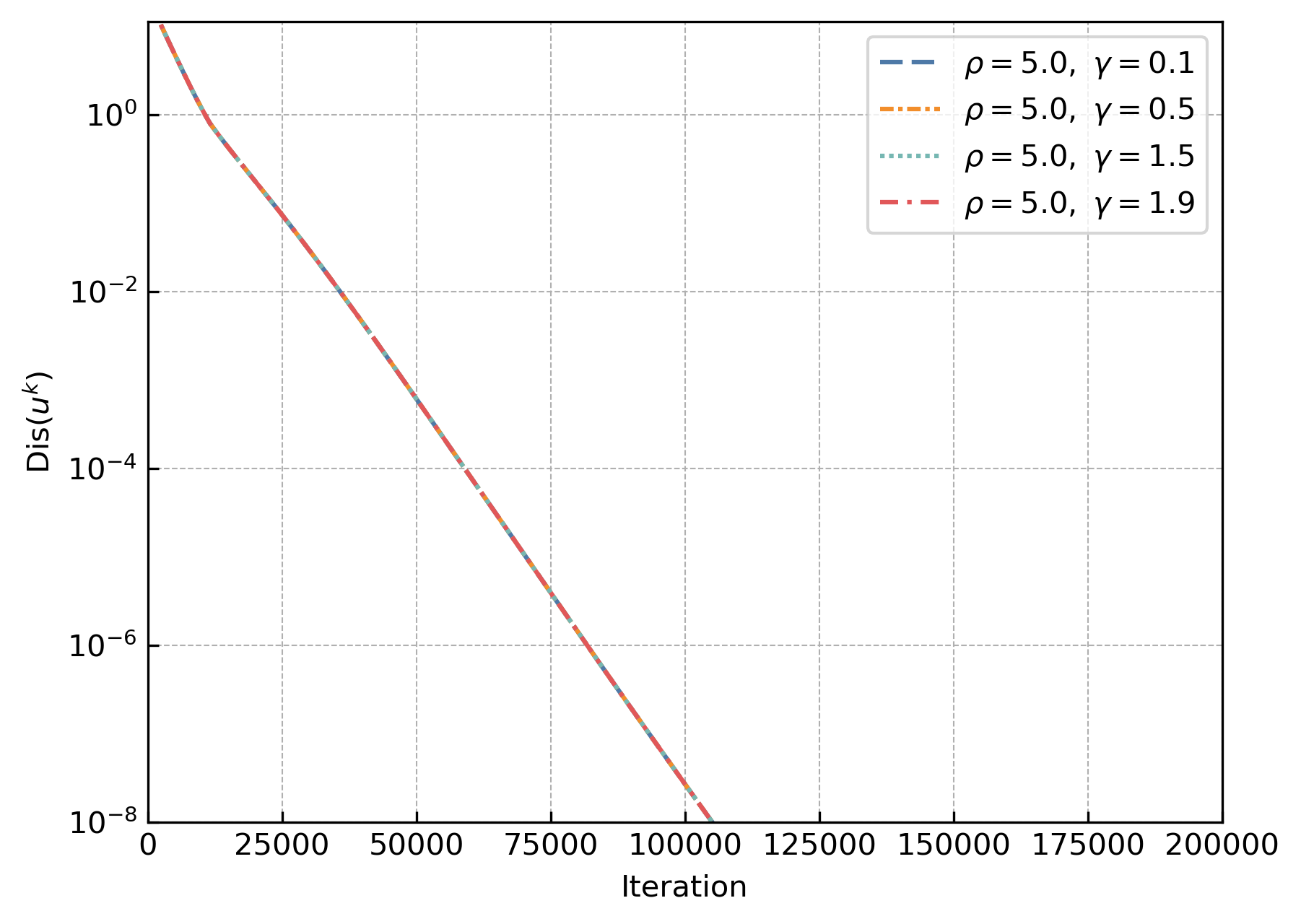} }}
    \caption{Experimental results on LCQP model for $N=10$ ($m=100$, $n=60$), with fixed $\rho$ and varying $\gamma$.}\label{fig-4}
\end{figure}

\subsection{The Optimal Resource Allocation Problem}\label{subsec-5}
In this example, we consider the following case:
\begin{equation*}
\begin{split}
&\min \sum_{i=1}^N f_i(x_i)
\\
&\textrm{s.t. } \sum_{i=1}^N x_i = 0
\end{split}
\end{equation*}
where $f_i:\mathbb{R} \rightarrow \mathbb{R}$ is given as $f_i(x_i) = (1/2)a_i(x_i-c_i)^2 + \log [1+exp(b_i(x_i-d_i))]$ with the coefficients $a_i, b_i,c_i,d_i$ generated randomly with uniform distributions on $[0,2]$, $[-2,2]$, $[-10,10]$, $[-10,10]$ respectively.
We conduct experiments with the damping parameter $\gamma \in \{0.1, 0.5, 1.5, 1.9\}$, and the penalty parameter $\rho \in \{0.03, 1, 5, 10\}$ in Algorithm \ref{algo}.
We update $x_i^{k+1}$ in \eqref{eq-1-5} using the equation derived from the optimality given as follows:
\begin{equation*}
a_i \big(x_i^{k+1}-c_i\big) + \frac{b_i exp\Big[b_i\big( x_i^{k+1}-d_i\big)\Big]}{1+exp\left[b_i\big( x_i^{k+1}-d_i\big)\right]} 
+\rho \bigg(x_i^{k+1} + \sum_{j\neq i}x_j^k - \frac{\lambda^k}{\rho} \bigg) + P_i\big(x_i^{k+1} - x_i^{k}\big)=0.
\end{equation*}

In order to study the convergence behavior of the test algorithms, we test two cases: $6$-block and $20$-block.
The value $u^* = \big(x_1^*, \ldots, x_N^* ; \lambda^*\big)$ is obtained numerically by running Algorithm \ref{algo} $4,000$ iterations.
We employ the same error metric $\text{dis}(u^k)$ as defined in \eqref{eq-4-1} for this experiment. 
The initial values $u^0 = (x_1^0, \ldots, x_N^0; \lambda^0)$ are set to zero.
The experimental results are plotted in Figures \ref{fig-5} and \ref{fig-6} for the $6$-block case, and in Figures \ref{fig-7} and \ref{fig-8} for the $20$-block case. 
We observe that the value $\text{dis}(u^k)$ converges linearly to zero, as anticipated in \mbox{Theorem \ref{thm-main}}.

Figure \ref{fig-5} illustrates that, for fixed damping parameters $\gamma \in \{0.1, 0.5, 1.5\}$, the algorithm achieves the fastest convergence when $\rho = 1$.
This suggests that simply choosing a smaller penalty parameter does not guarantee faster convergence, highlighting the need for proper choice of $\rho$ for the fast convergence.
In Figure \ref{fig-6}, the effect of the damping parameter $\gamma$ is investigated at fixed value of $\rho$.
When $\rho = 0.03$, increasing $\gamma$ leads to improved convergence rates. 
In contrast, for $\rho = 5$, the convergence rates remain similar regardless of the value of $\gamma$, implying that the algorithm's performance is relatively insensitive to changes in $\gamma$.

\begin{figure}[htbp]
    \centering
    \subfloat[$\rho = \{0.03, 1, 5, 10\}$, $\gamma = 0.1$]{{\includegraphics[width=6cm]{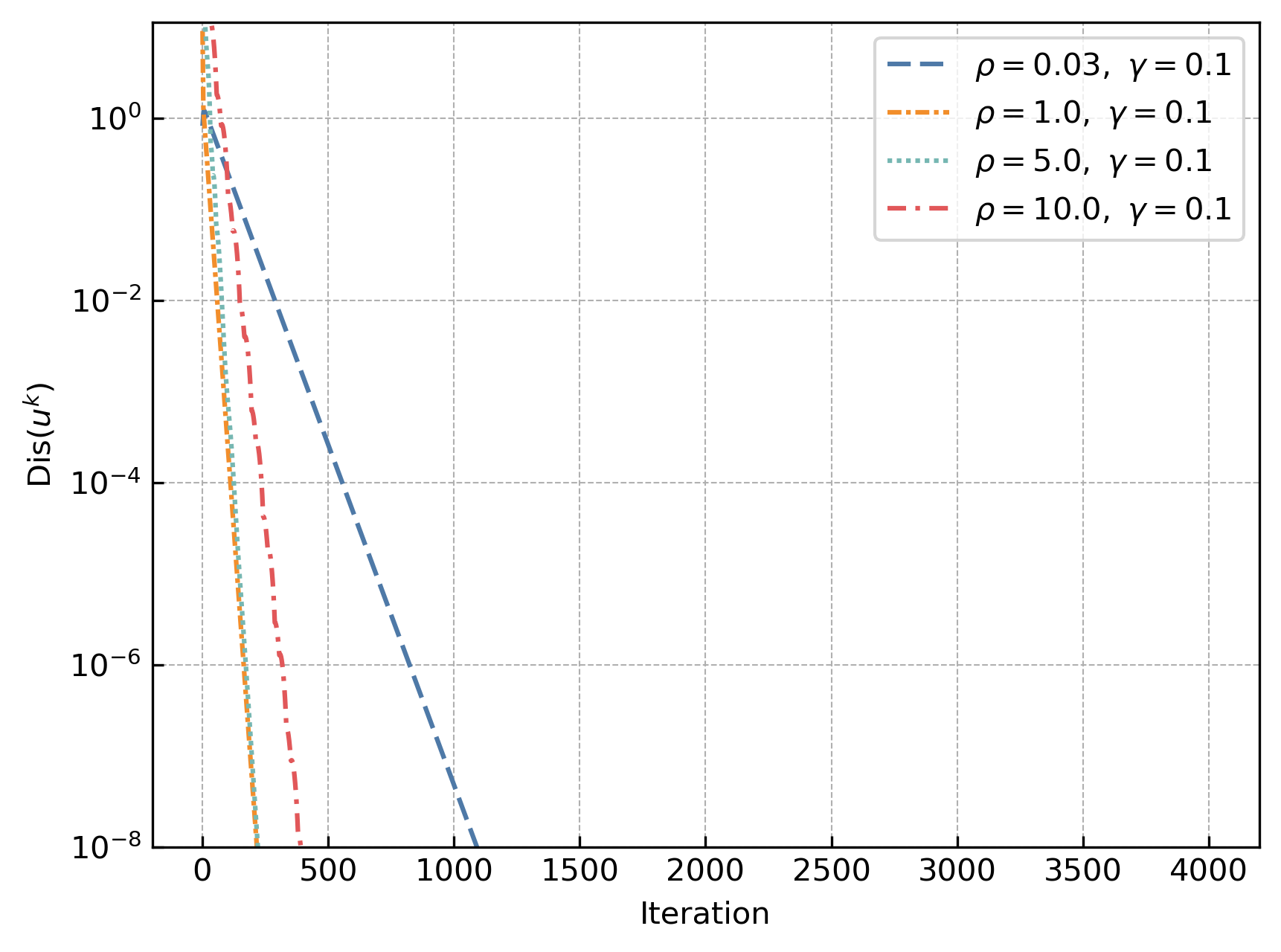} }}
    \subfloat[$\rho = \{0.03, 1, 5, 10\}$, $\gamma = 0.5$]{{\includegraphics[width=6cm]{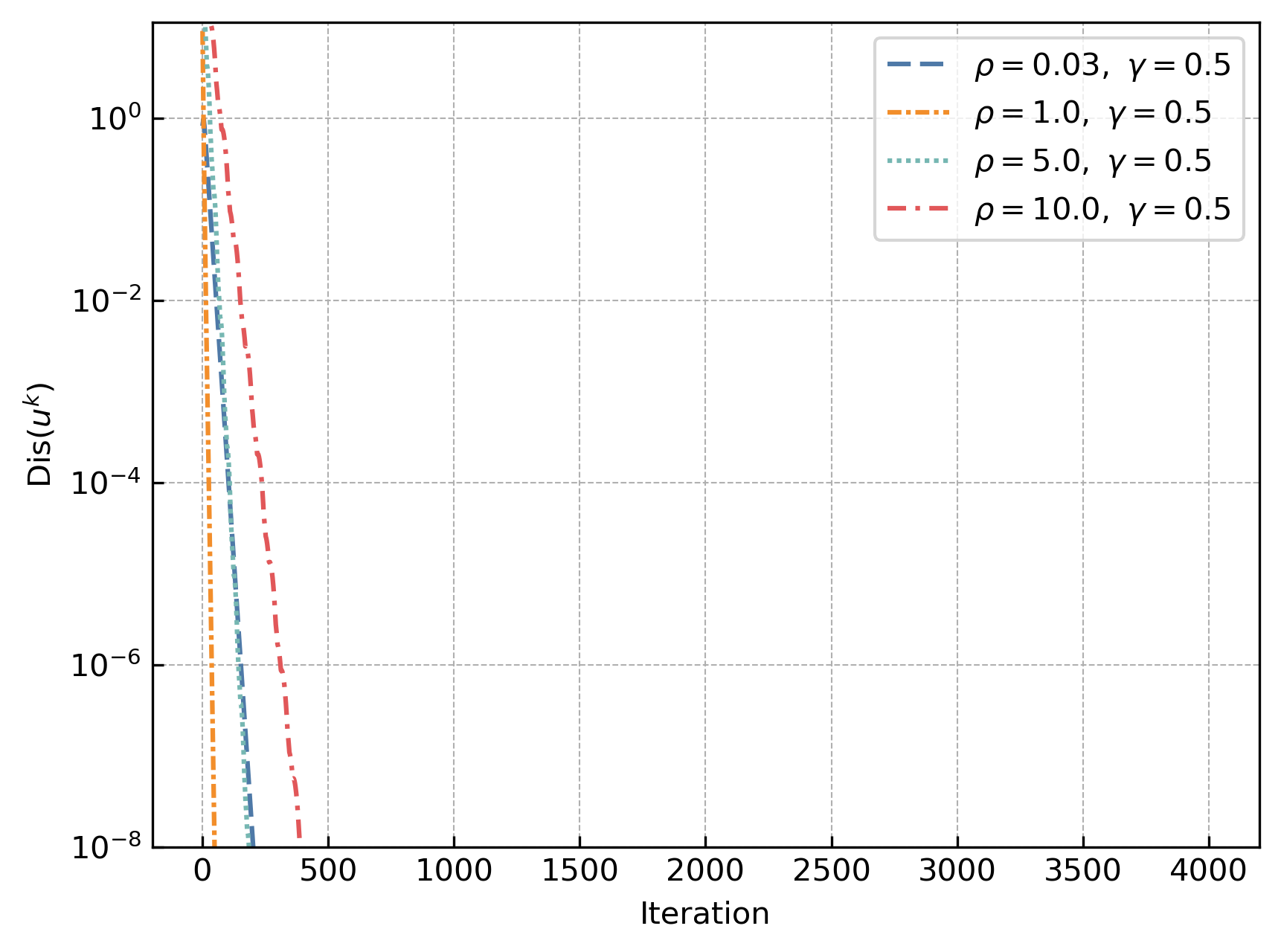} }}
    \\
    \subfloat[$\rho = \{0.03, 1, 5, 10\}$, $\gamma = 1.5$]{{\includegraphics[width=6cm]{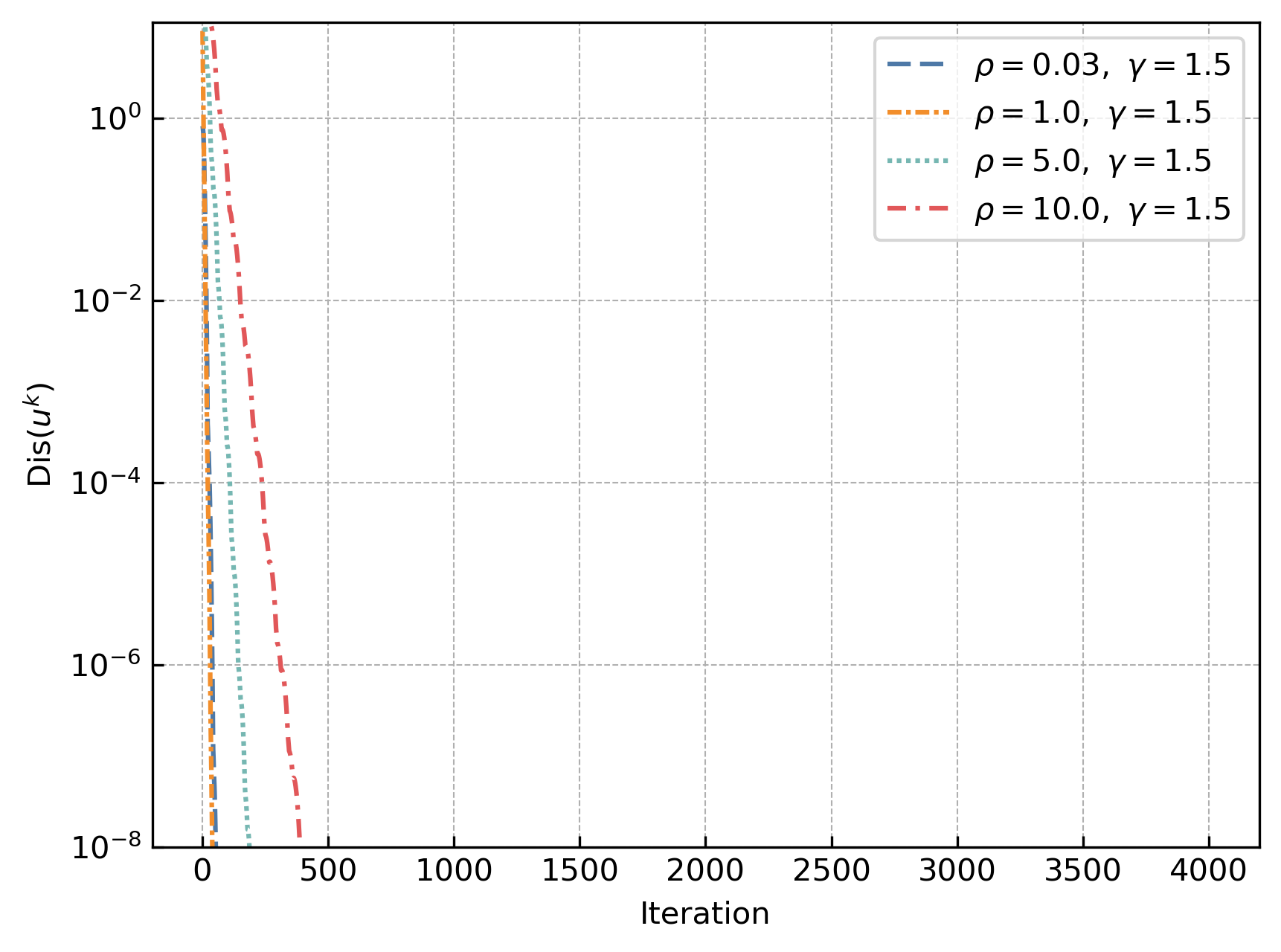} }}
    \subfloat[$\rho = \{0.03, 1, 5, 10\}$, $\gamma = 1.9$]{{\includegraphics[width=6cm]{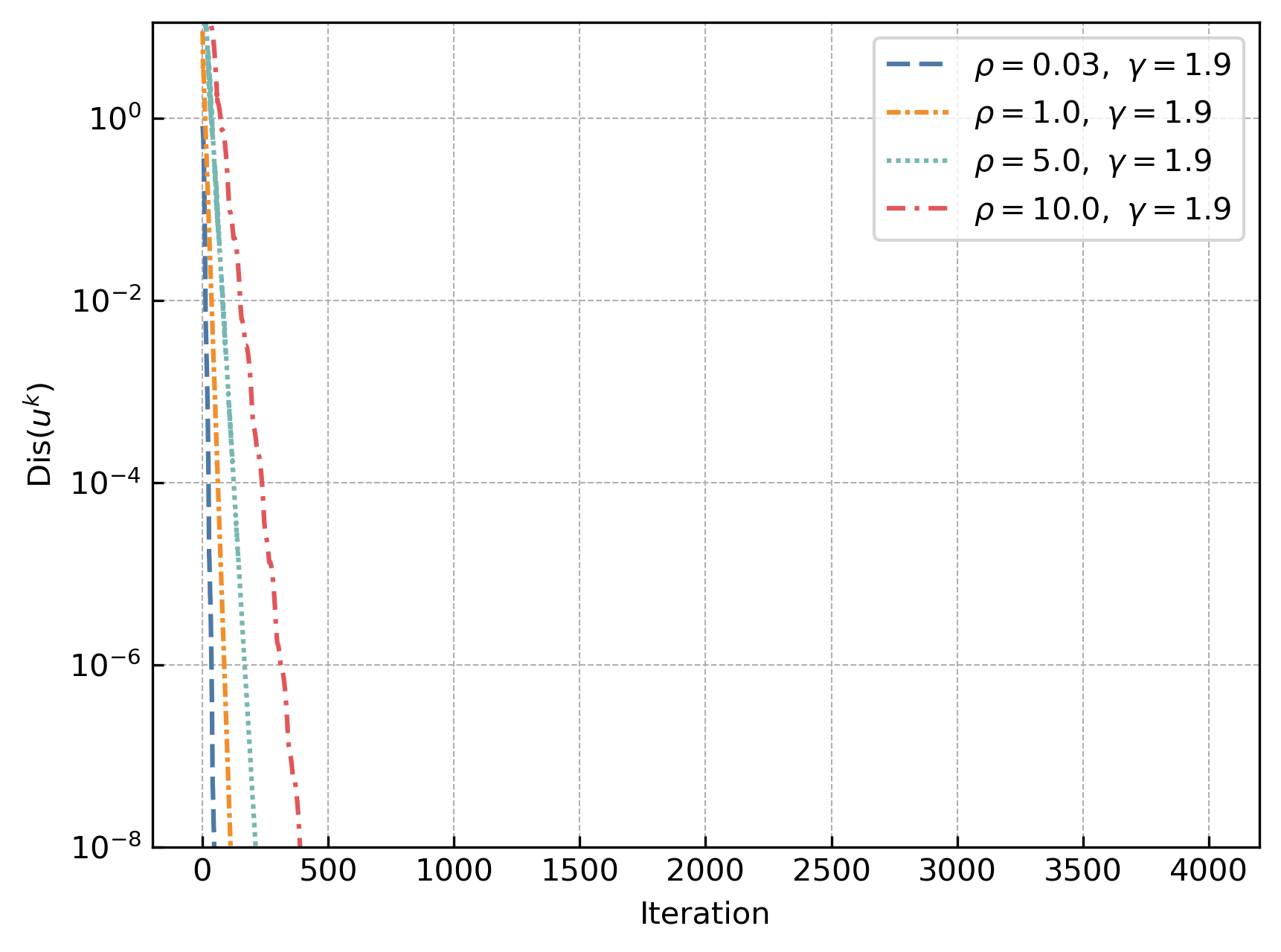} }}
    \caption{Experimental results on the optimal resource allocation model for $N=6$, with fixed $\gamma$ and varying $\rho$.}\label{fig-5}
\end{figure}

\begin{figure}[htbp]
    \centering
    \subfloat[$\rho = 0.03$, $\gamma = \{0.1, 0.5, 1.5, 1.9\}$]{{\includegraphics[width=6cm]{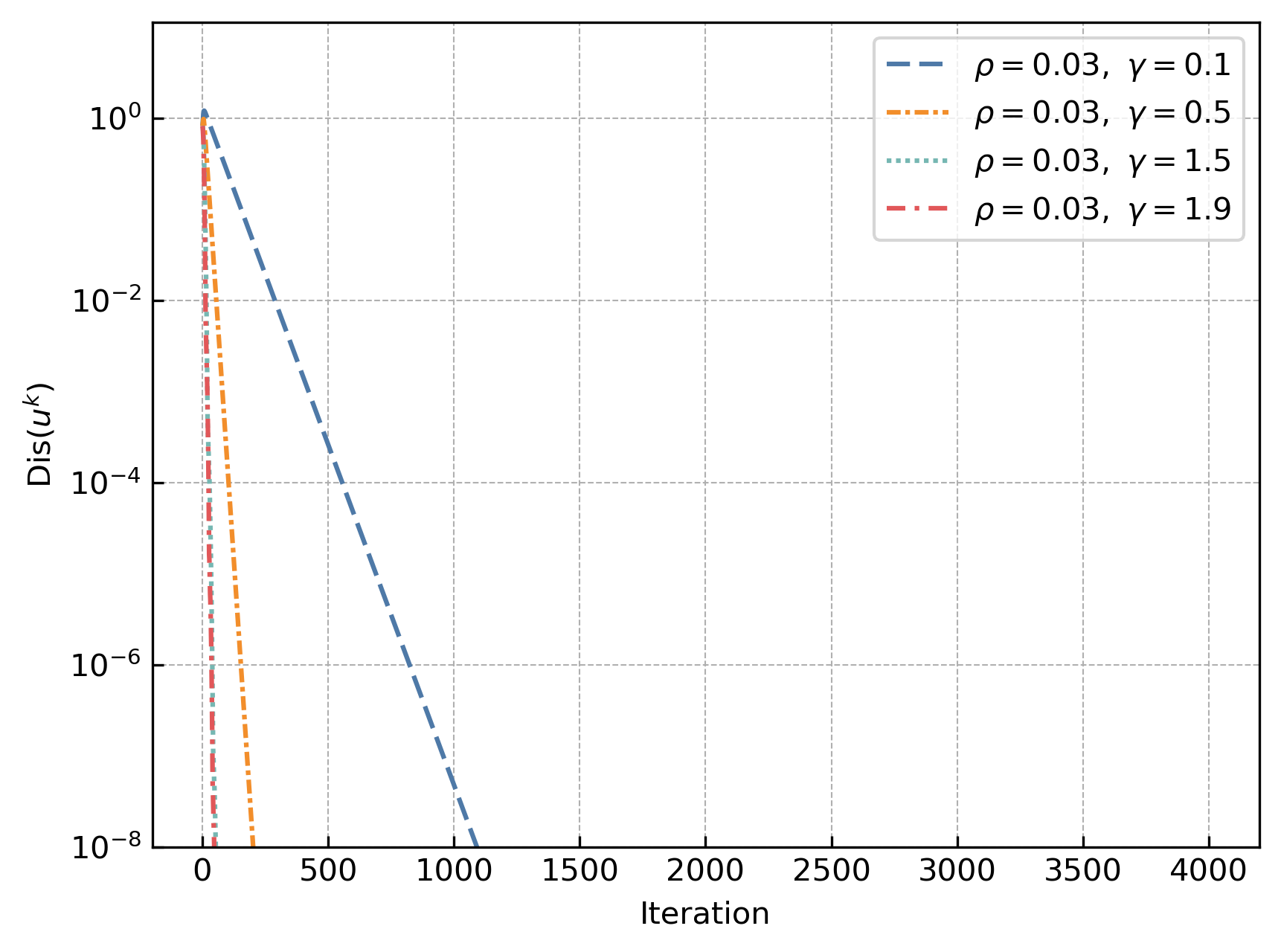} }}
    \subfloat[$\rho = 5$, $\gamma = \{0.1, 0.5, 1.5, 1.9\}$]{{\includegraphics[width=6cm]{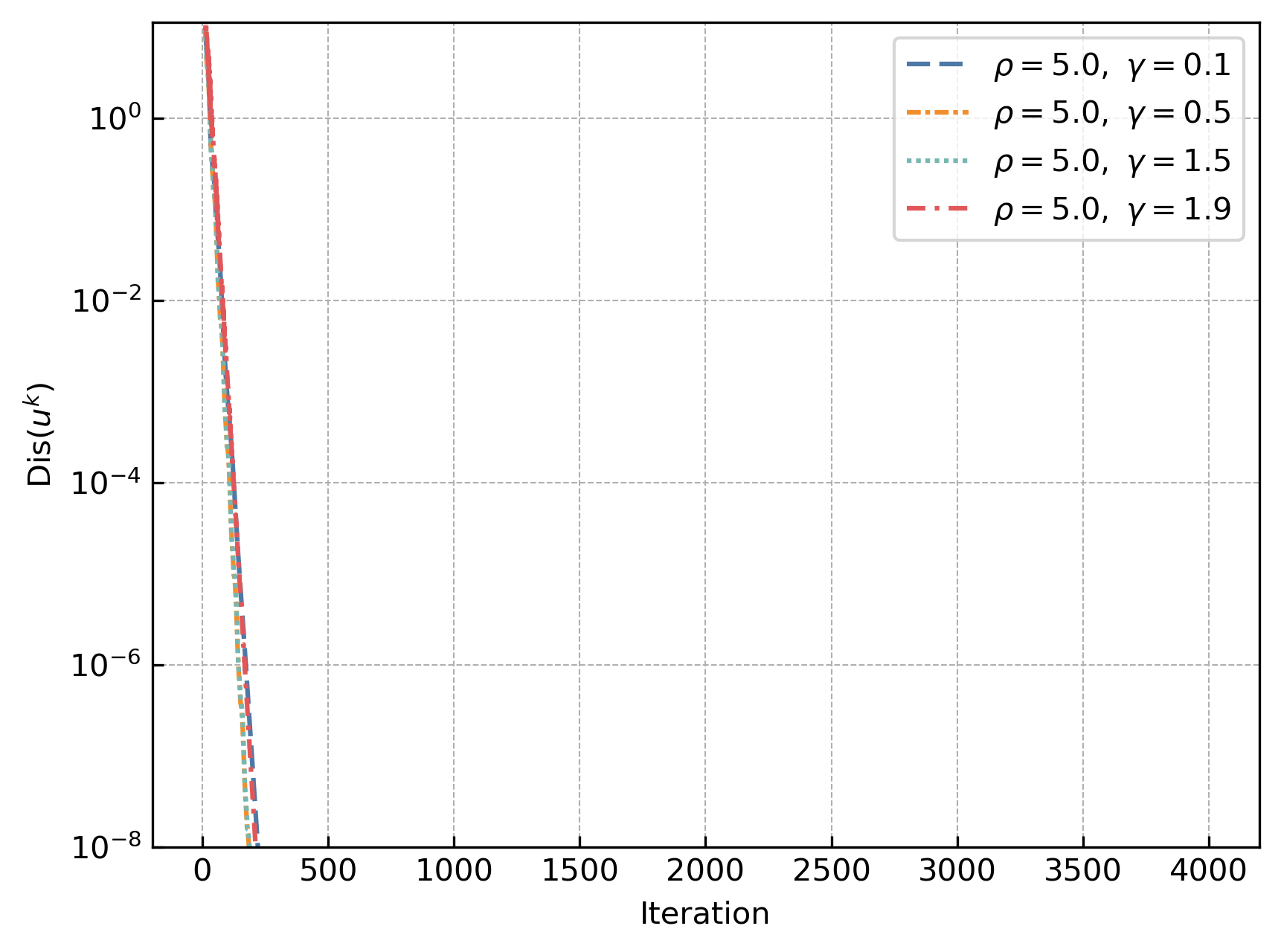} }}
    \caption{Experimental results on the optimal resource allocation model for $N=6$, with fixed $\rho$ under varying $\gamma$.}\label{fig-6}
\end{figure}

\begin{figure}[htbp]
    \centering
    \subfloat[$\rho = \{0.03, 1, 5, 10\}$, $\gamma = 0.1$]{{\includegraphics[width=6cm]{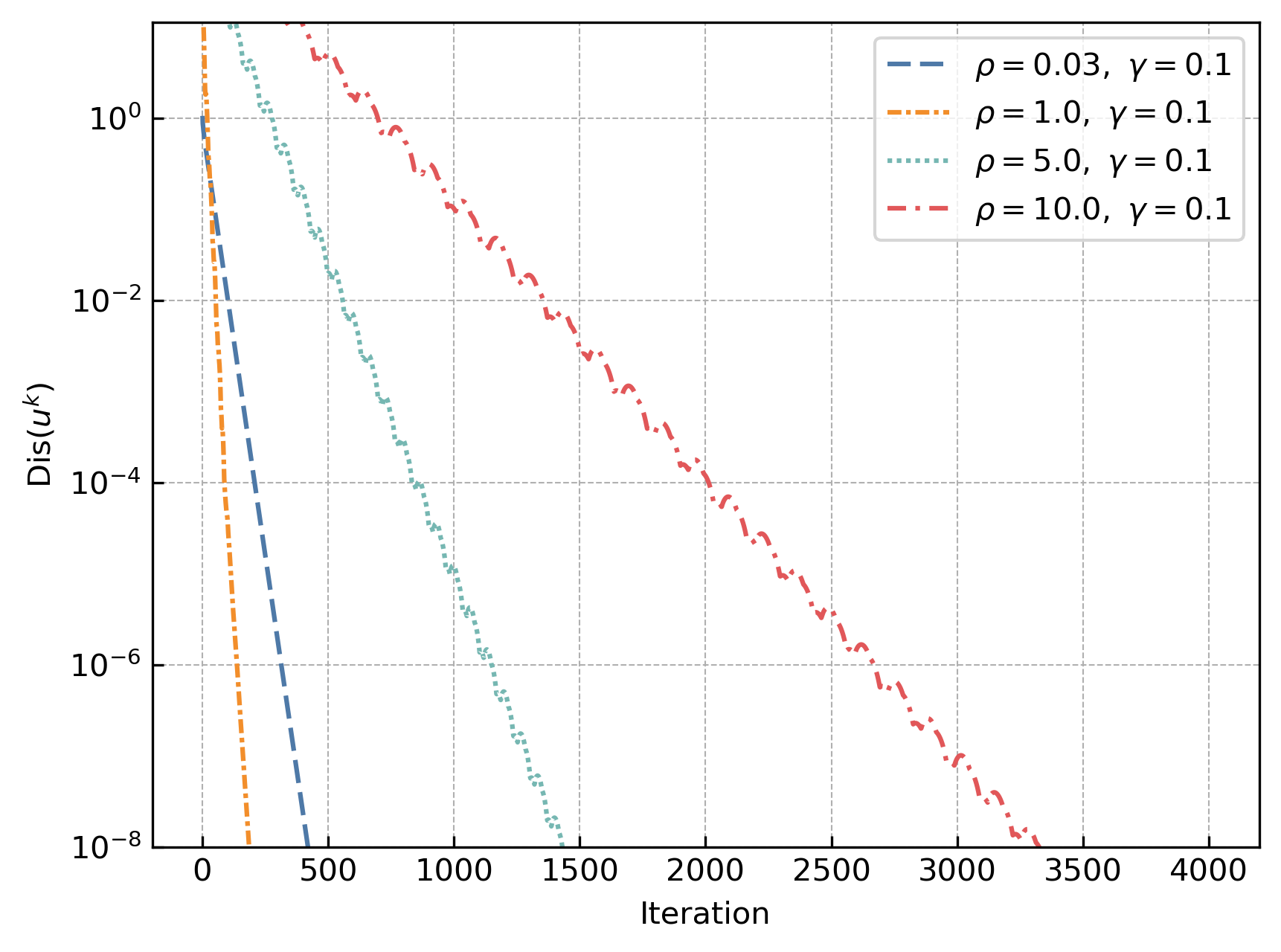} }}
    \subfloat[$\rho = \{0.03, 1, 5, 10\}$, $\gamma = 0.5$]{{\includegraphics[width=6cm]{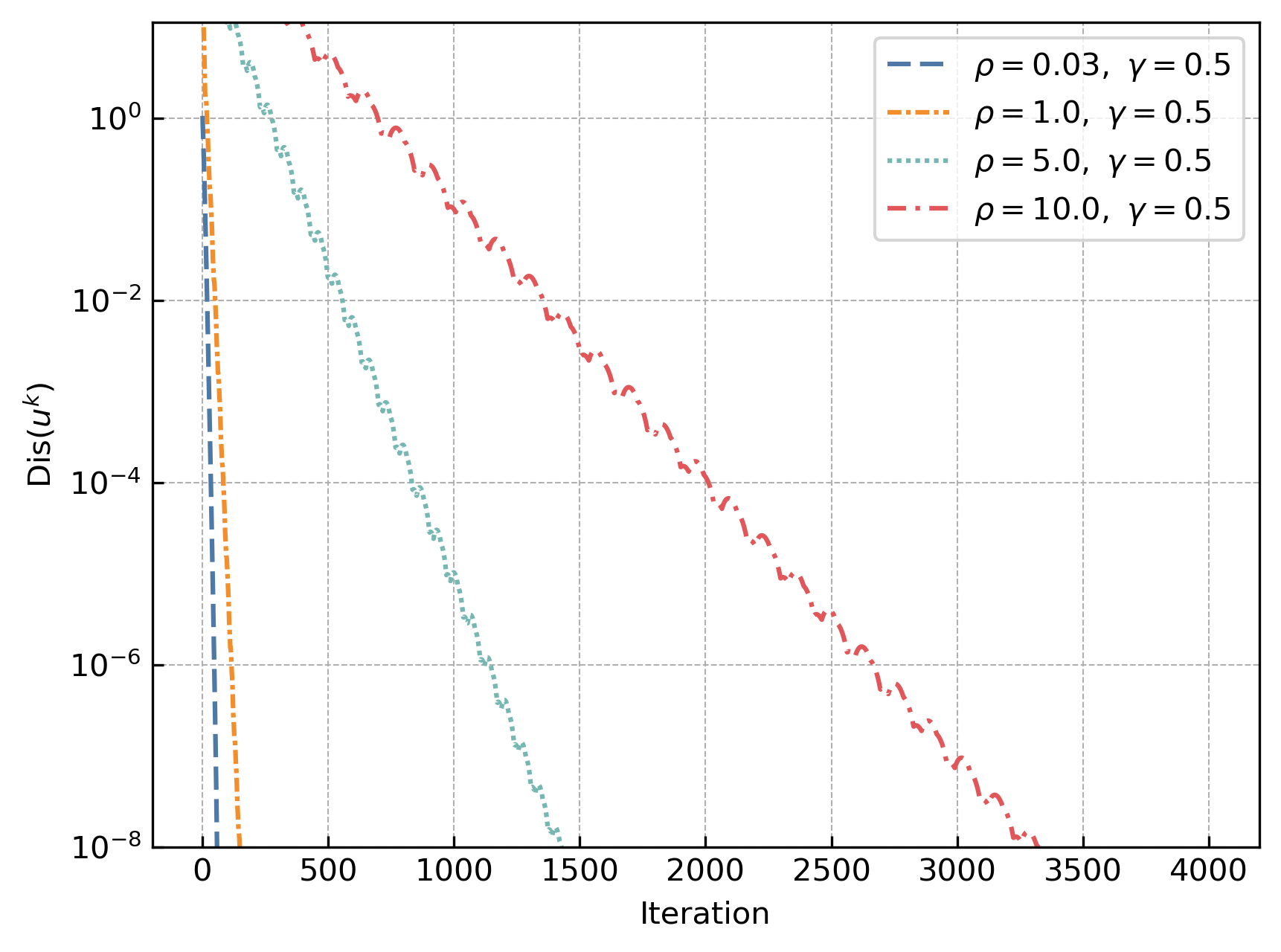} }}
    \\
    \subfloat[$\rho = \{0.03, 1, 5, 10\}$, $\gamma = 1.5$]{{\includegraphics[width=6cm]{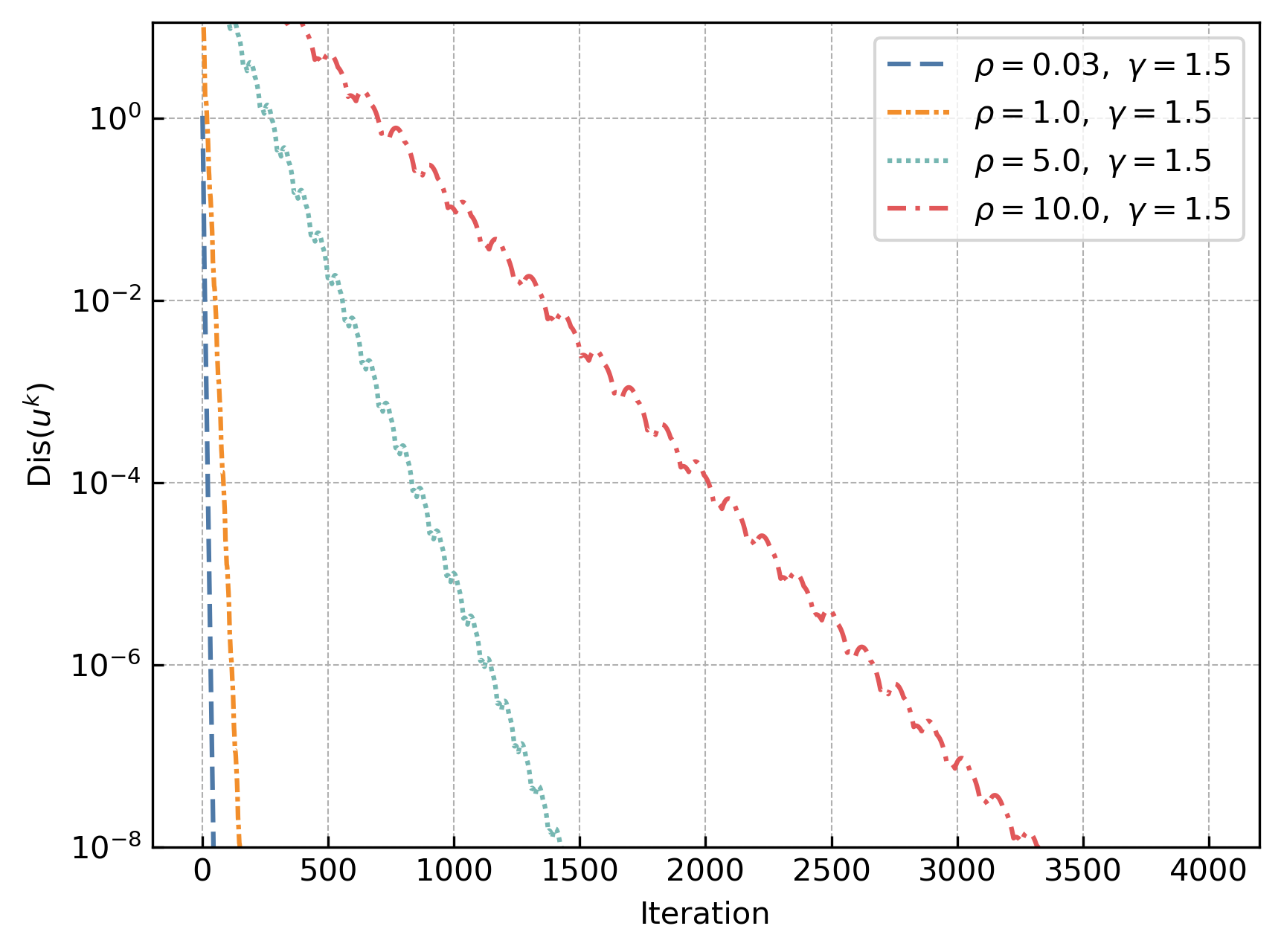} }}
    \subfloat[$\rho = \{0.03, 1, 5, 10\}$, $\gamma = 1.9$]{{\includegraphics[width=6cm]{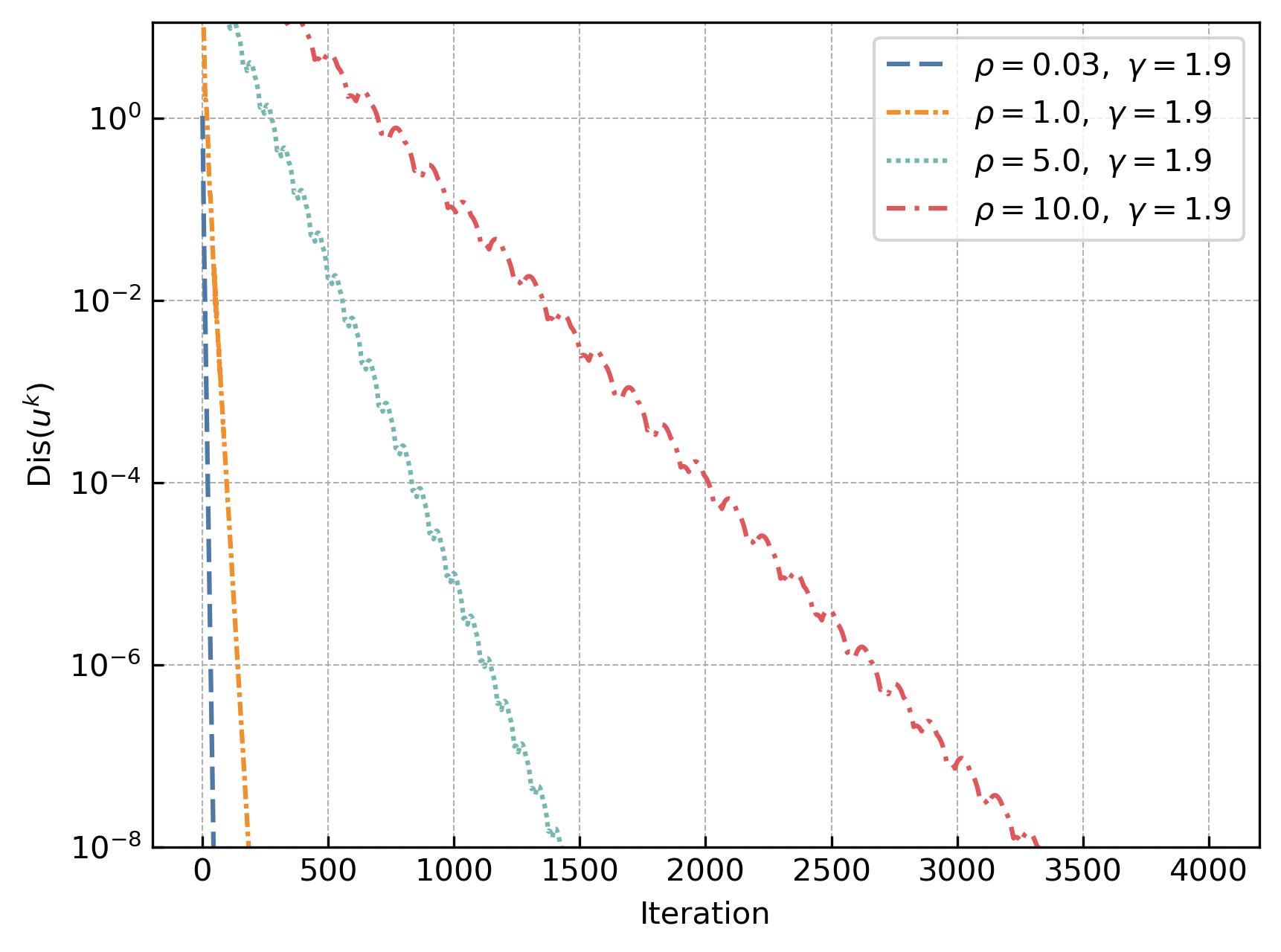} }}
    \caption{Experimental results on the optimal resource allocation model for $N=20$, with fixed $\gamma$ under varying $\rho$.}\label{fig-7}
\end{figure}

\begin{figure}[htbp]
    \centering
    \subfloat[$\rho = 0.03$, $\gamma = \{0.1, 0.5, 1.5, 1.9\}$]{{\includegraphics[width=6cm]{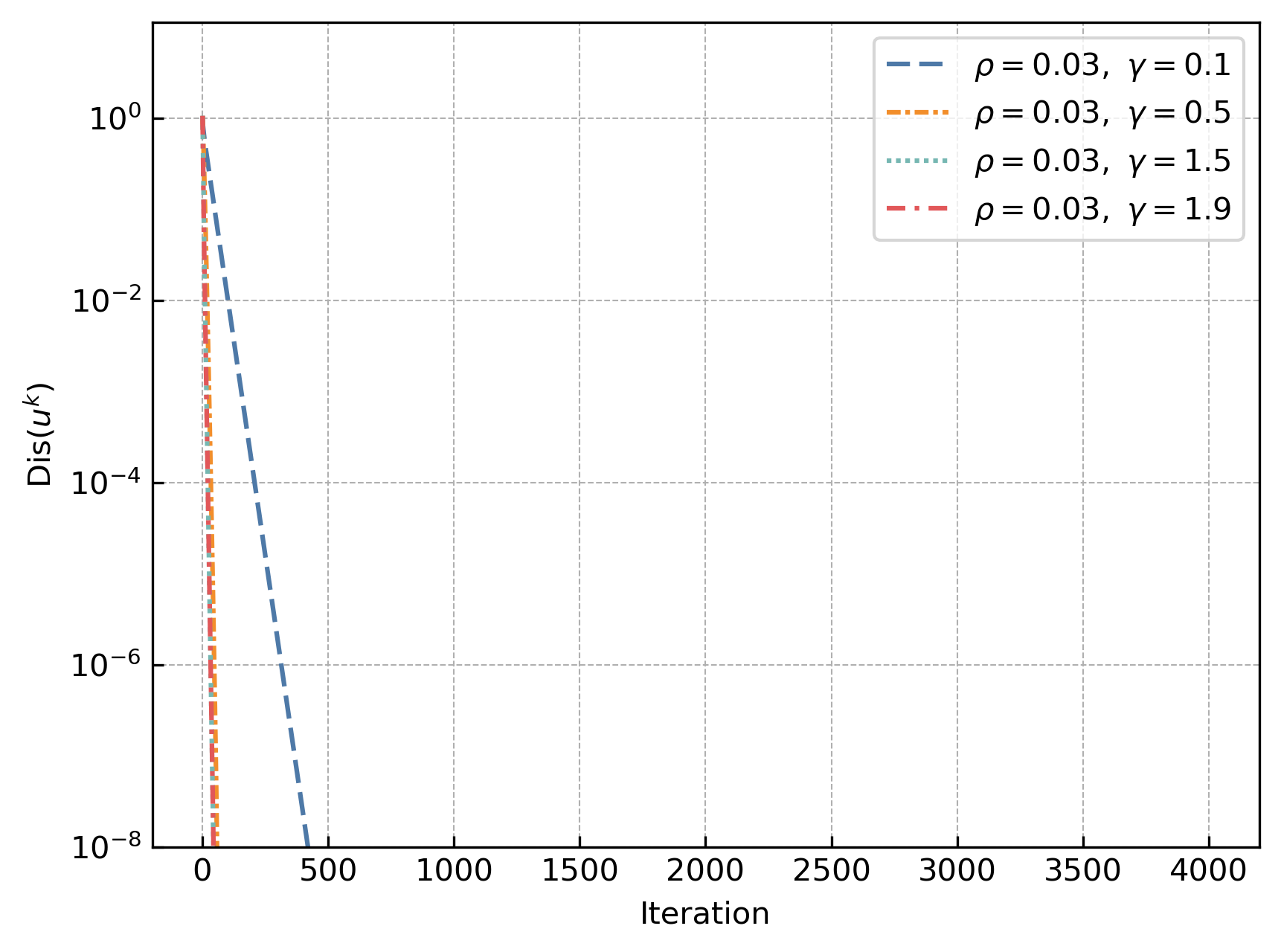} }}
    \subfloat[$\rho = 5$, $\gamma = \{0.1, 0.5, 1.5, 1.9\}$]{{\includegraphics[width=6cm]{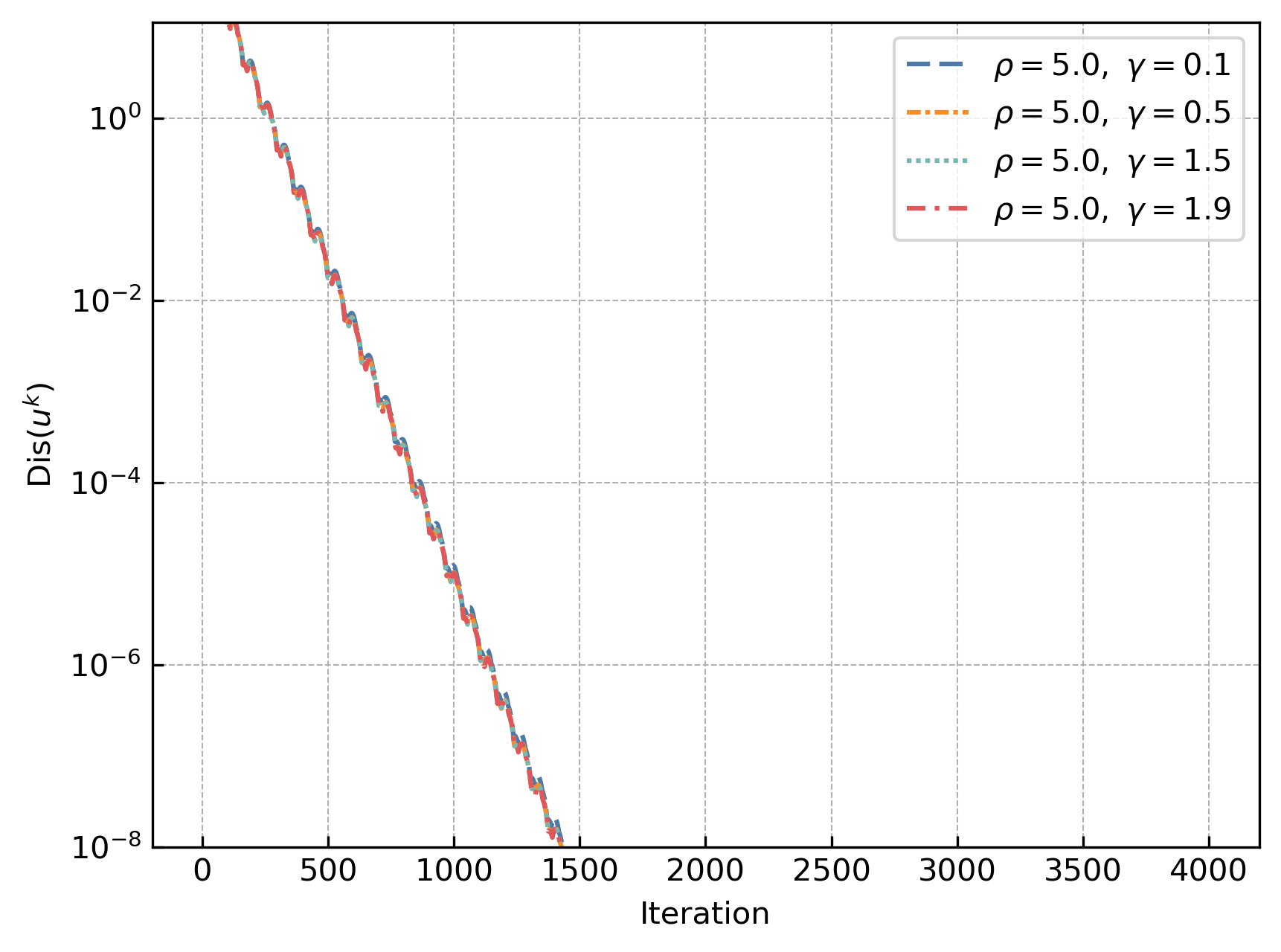} }}
    \caption{Experimental results on the optimal resource allocation model for $N=20$, with fixed $\rho$ under varying $\gamma$.}\label{fig-8}
\end{figure}

\section{Conclusion}\label{sec-5}
In this paper, we study the Jacobi-Proximal ADMM algorithm designed to solve the linearly coupled optimization problems in a parallel way.
We achieved the first result for the linear convergence of the algorithm and derived its convergence rate when the cost functions are strongly convex and Lipschitz continuous.
This is the first linear convergence result for the Jacobi-Proximal ADMM.
We also presented numerical experiments supporting the convergence result.

\section*{Declarations}
\subsection*{Funding}
No funding was received for conducting this study.

\subsection*{Competing interests}
The authors have no relevant financial or non-financial interests to disclose.

\section*{Code availability}
The Python code used for the numerical experiments in Section \ref{sec-4} is publicly available at \url{https://github.com/hyelinchoi/Jacobi_Proximal_ADMM}.

\bibliography{bib}

\end{document}